\RequirePackage{fix-cm}

 \documentclass[smallextended]{svjour3} 
 \smartqed  

\usepackage{amssymb,amsmath,amsfonts,latexsym} 
\usepackage{hyperref}
   
 \def\proofend{\hfill$\Box$} 

\usepackage{mathtools}
\usepackage{multirow,array}
\usepackage{graphicx}
\usepackage{subfigure}
\usepackage{epstopdf}
\usepackage{float} 
\usepackage{color, xcolor}
\allowdisplaybreaks

\journalname{Fract. Calc. Appl. Anal.} 

\begin{document}

\title{Qualitative properties of solutions to a nonlinear time-space fractional diffusion equation
}

\titlerunning{Qualitative properties of solutions to a nonlinear ...}        

\author{Meiirkhan B. Borikhanov \and  Michael Ruzhansky         \and Berikbol T. Torebek 
}

\authorrunning{M. B. Borikhanov et al.} 

\institute{Meiirkhan B. Borikhanov \at Khoja Akhmet Yassawi International Kazakh--Turkish University, Sattarkhanov ave., 29, 161200 Turkistan, Kazakhstan \\Institute of
Mathematics and Mathematical Modeling, 125 Pushkin str., 050010 Almaty, Kazakhstan \at
            \email{meiirkhan.borikhanov@ayu.edu.kz}           
           \and
  Michael Ruzhansky\at
Department of Mathematics: Analysis, Logic and Discrete Mathematics,
Ghent University, Ghent, Belgium
\\School of Mathematical Sciences, Queen Mary University of London, London, United Kingdom  \at 
\email{michael.ruzhansky@ugent.be}
\and
Berikbol T. Torebek$^*$ \at
Institute of Mathematics and Mathematical Modeling, 125 Pushkin str.,
050010 Almaty, Kazakhstan \\ Department of Mathematics: Analysis, Logic and Discrete Mathematics,
Ghent University, Ghent, Belgium \at
\email{berikbol.torebek@ugent.be, $^*$ corresponding author}}

\date{Received: 28 July 2022 / Revised: 14 November 2022 / Accepted: 24 November 2022}

\maketitle
\begin{abstract}In the present paper, we study the Cauchy-Dirichlet problem to a nonlocal nonlinear
diffusion equation with polynomial nonlinearities
$$\mathcal{D}_{0|t}^{\alpha }u+(-\Delta)^s_pu=\gamma|u|^{m-1}u+\mu|u|^{q-2}u,\,\gamma,\mu\in\mathbb{R},\,m>0,q>1,$$ 
involving time-fractional Caputo derivative $\mathcal{D}_{0|t}^{\alpha}$ and space-fractional $p$-Laplacian operator $(-\Delta)^s_p$. 

We give a simple proof of the comparison principle for the considered problem using purely algebraic relations, for different sets of $\gamma,\mu,m$ and $q$.

The Galerkin approximation method is used to prove the existence of a local weak solution. The blow-up phenomena, existence of global weak solutions and asymptotic behavior of global solutions are classified using the comparison principle.

\end{abstract}

\section{Introduction}
\setcounter{section}{1} \setcounter{equation}{0}
In this paper, we study the initial-boundary value problem for the nonlinear time-space fractional diffusion equation 
\begin{equation}\label{01}
\left\{\begin{array}{l}
\mathcal{D}_{0|t}^{\alpha }u+(-\Delta)^s_pu=\gamma|u|^{m-1}u+\mu|u|^{q-2}u,\, \,(x,t)\in\Omega\times(0,T),\\{}\\
u(x,t)=0,\,x\in \mathbb{R}^N\setminus\Omega,\,t\in(0,T),\\{}\\
u(x,0)=u_0(x),\, x\in\Omega,\end{array}\right.\end{equation}
where $\Omega\subset \mathbb{R}^N$  is a smoothly bounded domain; $s\in(0,1), p\geq2, m>0, q\geq1$,  $\gamma,\mu\in \mathbb{R}$  and  $\mathcal{D}_{0|t}^{\alpha }$ is the left Caputo fractional derivative  of order $\alpha\in(0,1)$ (see Definition \ref{CD}). 

In recent years, the study of differential equations using non-local fractional operators has attracted a lot of interest. The time-space fractional diffusion equations could be applied to a wide range of applications, including finance, semiconductor research, biology and hydrogeology, continuum mechanics, phase transition phenomena, population dynamics, image process, game theory and Lévy processes, (see \cite{Lei}, \cite{Bertoin}, \cite{Caffarelli}, \cite{Laskin}, \cite{Gilboa}, \cite{Gal}, \cite{Andrade}) and the references therein. When a particle flow  spreads at a rate that defies Brownian motion theories, both time and spatial fractional derivatives (see \cite{Giga}, \cite{Nane}, \cite{Vazquez1}) can be employed to simulate anomalous diffusion or dispersion. 
Recently, motivated by some situations arising in the game theory, nonlinear generalizations of the fractional Laplacian have been introduced, (see \cite{Bjorland}, 
\cite{Caffarelli}).

Later on, the fractional version of the $p$-Laplacian was studied through energy and test function methods by Chambolle  and al. in \cite{Chambolle}. The viscosity version of this non-local operator was given by Ishii and al. in \cite{Ishii},  Bjorland and al. in \cite{Bjorland}.

In the case $\alpha=s=1,\,\gamma=-1$ the problem \eqref{01} coincides with a quasilinear parabolic equation which  has been studied by Li et al in \cite{Li}. By using a Gagliardo-Nirenberg type inequality, the energy method and
comparison principle, the phenomena of blow-up and extinction have been classified completely in the different ranges of reaction exponents.

Moreover, when $\alpha=s=1, m>1$ and the coefficients are $\gamma>0, \mu=0$, the problem \eqref{01} was considered by Yin and Jin in \cite{Yin}.
They determined the critical extinction and blow-up exponents for the homogeneous Dirichlet boundary value problem.

Vergara and Zacher in \cite{Vergara}  have considered nonlocal in time semilinear subdiffusion equations on a bounded domain, 
\begin{equation}\label{Z}
\left\{\begin{array}{l}
\mathcal{D}_{0|t}^{\alpha } u -\operatorname{div}(A(x,t)\nabla u)=f(u),\,x\in\Omega,\,\,t>0,\\
u(x,t)=0,\,\, x\in\partial\Omega,\,\,t>0, \\
u(x,0)=u_0(x),\, x\in\Omega,\end{array}\right.\end{equation}
where the coefficients $A=(a_{ij})$ were assumed to satisfy $$\left(A(x,t)\xi,\xi\right)\geq \nu|\xi|^2,\,\,\text{for a.e.}\,\,\,(x,t)\in \Omega\times(0,+\infty)\,\,\text{and all}\,\,\, \xi \in \mathbb{R}^N.$$ 
They proved a well-posedness
result in the setting of bounded weak solutions and studied the stability and instability of the zero function in
the special case where the nonlinearity vanishes at 0. In addition, they established a blow-up result for positive convex
and superlinear nonlinearities.

Later on, Alsaedi et al. \cite{Alsaedi} have studied the KPP-Fisher-type reaction-diffusion equation, which is the problem \eqref{01} in the case $p=2, \gamma=-1, q=3$ and $\mu=m=1$,  in a bounded domain. 
Under some conditions on the initial data, they have showed that solutions may experience blow-up in a finite time. However, for realistic initial conditions, solutions are global in time. Moreover, the asymptotic behavior of bounded solutions was analysed.

Recently, in \cite{Tuan}, Tuan, Au and Xu studied the initial-boundary value problem for the fractional pseudo-parabolic equation with
fractional Laplacian
\begin{equation}\label{PsPF}
\left\{\begin{array}{l}
\mathcal{D}_{0|t}^{\alpha} (u-m \Delta u)+(-\Delta)^s u=\mathcal{N}(u),  \, x\in\Omega,\,t>0,\\
u(x,t)=0,\, x\in\partial\Omega, t>0,\\
u(x,0)=u_0(x),\, x\in\Omega,\end{array}\right.\end{equation}
where $s\in(0,1), m>0$ is a constant, and $\mathcal{N}(u)$ is the source
term satisfying one of the following conditions:
\begin{description}
\item[(a)] $\mathcal{N}(u)$ is a globally Lipschitz function;
\item[(b)] $\mathcal{N}(u)=|u|^{p-2}u,\,p\geq 2;$
\item[(c)] $\mathcal{N}(u)=|u|^{p-2}u\log |u|,\,\,p\geq 2.$
\end{description}
For the above cases, they proved the existence of a unique local mild solution and finite time blow-up solution to equation \eqref{PsPF}.
Because of the nonlocality of the equation, the authors believe that proving the existence of a weak solution using the Galerkin method for equation \eqref{PsPF} is problematic.

Motivated by the above results, in this paper we consider the time and space fractional quasilinear parabolic equation \eqref{01}. 

Using the Galerkin method, we prove the existence of a local weak solution to problem \eqref{01}.

This, in turn, partially answers the question posed in \cite{Tuan} about the existence of a local weak solution to the fractional pseudo-parabolic equation. In addition, a comparison principle to problem \eqref{01} is obtained, and we have investigated results on the blow-up and global solution using this concept.

\section{Preliminaries}
\setcounter{section}{2} \setcounter{equation}{0}
\subsection{The fractional Sobolev space} In this subsection, let us  recall some necessary definitions and useful properties of the fractional Sobolev space.

Let $s\in(0,1)$ and $p\in[1,+\infty)$ be real numbers, and let the fractional critical exponent be defined as $\large \displaystyle p_c^*=\frac{Np}{N-sp}$ if $sp<N$ or $p_c^*=\infty$, otherwise. 

One defines the fractional Sobolev space as follows 
\begin{equation*}
W^{s,p}(\mathbb{R}^N):=\biggl\{u\in L^p(\mathbb{R}^N), \frac{|u(x)-u(y)|}{|x-y|^{\frac{N}{p}+s}}\in L^p(\mathbb{R}^N\times\mathbb{R}^N)\biggr\}.    
\end{equation*}

This is the Banach space between $L^p(\mathbb{R}^N)$ and $W^{1,p}(\mathbb{R}^N)$, endowed with the norm
$$\|u\|_{W^{s,p}(\mathbb{R}^N)}:=\|u\|_{L^p{(\mathbb{R}^N})} +\biggl(\int_{\mathbb{R}^N}\int_{\mathbb{R}^N}\frac{|u(x)-u(y)|^p}{|x-y|^{N+sp}}dxdy\biggr)^\frac{1}{p}.$$

Let $\Omega$ be an open set in $\mathbb{R}^N$ and let $\mathcal{W}=(\mathbb{R}^N\times\mathbb{R}^N)\backslash ((\mathbb{R}^N\backslash\Omega)\times(\mathbb{R}^N\backslash\Omega))$. It is obvious that $\Omega\times\Omega$ is strictly contained in $\mathcal{W}$.

Denote
\begin{equation*}
W^{s,p}(\Omega):=\biggl\{u\in L^p(\Omega),\, u=0\,\,  \text{in}\,\, \mathbb{R}^N\backslash\Omega,\,  \frac{|u(x)-u(y)|}{|x-y|^{\frac{N}{p}+s}}\in L^p(\mathcal{W})\biggr\}.    
\end{equation*}

The space $W^{s,p}(\Omega)$ is also endowed with the norm
$$\|u\|_{W^{s,p}(\Omega)}:=\|u\|_{L^p(\Omega)} +[u]_{W^{s,p}(\Omega)},$$
where the term
$$[u]_{W^{s,p}(\Omega)}:=\biggl(\int_\Omega\int_\Omega\frac{|u(x)-u(y)|^p}{|x-y|^{N+sp}}dxdy\biggr)^\frac{1}{p}<\infty$$
is the so-called Gagliardo semi-norm of $u$, which was introduced by Gagliardo \cite{Gagliardo} to describe the trace spaces of Sobolev maps.  

We refer to \cite{Valdinoci} and \cite{Brasco1}, where one can find a description of the most useful properties of the fractional Sobolev spaces $W^{s,p}(\Omega)$. In the literature, fractional Sobolev-type spaces are also called \textit{Aronszajn, Gagliardo} or \textit{Slobodeckij spaces}, by the name of the people who first introduced them, practically concurrently (see \cite{Aronszajn,Gagliardo1,Slobodeckij}).

For Gagliardo semi-norms, the next result is a Poincar\'{e} inequality. This is standard, but we should also always pay careful attention to the sharp constants  dependence on $s$.

\begin{lemma}[\cite{Valdinoci}, Theorem 6.7]\label{EMBD} Let $s\in(0,1)$ and $p\in[1,+\infty)$ be such that $sp<N$. Let $\Omega\subseteq \mathbb{R}^N$ be an extension domain for $W^{s,p}$. Then, there exists a positive constant $C=C(N,p,s,\Omega)$ such that, for any $u\in W^{s,p}(\Omega)$, we have
$$\|u\|_{L^q(\Omega)}\leq C\|u\|_{W^{s,p}(\Omega)},$$
for any $q\in[p,p^*],$ where $\large\displaystyle p^*=p^*(N,s)=\frac{Np}{N-sp}$ is the so-called fractional critical exponent. 

That means, the space $W^{s,p}(\Omega)$ is continuously embedded in $L^q(\Omega)$ for $q\in[p,p^*].$ If, in addition, $\Omega$ is bounded, then, the space $W^{s,p}(\Omega)$ is continuously embedded in $L^q(\Omega)$ for $q\in[1,p^*].$
\end{lemma}

\begin{lemma}[\cite{Nguyen}, Lemma 2.1. Fractional Gagliardo–Nirenberg inequality]\label{FCKN} Let $p>1, \tau>0, N,q\geq1, 0<s<1$ and $0 <a\leq1$ be such that $$\frac{1}{\tau}=a\left(\frac{1}{p}-\frac{s}{N}\right)+\frac{1-a}{q}.$$

We have 
$$\|u\|_{L^\tau(\mathbb{R}^N)}\leq C[u]^a_{W^{s,p}(\mathbb{R}^N)}\|u\|^{(1-a)}_{L^q(\mathbb{R}^N)},\,\,\text{for}\,\,\,u\in C^1_c(\mathbb{R}^N),$$for some positive constant $C$ independent of $u$.\end{lemma}

\subsection{Fractional operators} This part is devoted to the definitions and properties of fractional derivatives in time and space.

\begin{definition}[\cite{Kilbas}, p. 69]\label{RD} The left and right Riemann-Liouville fractional integrals of order $0<\alpha<1$ for an integrable function $u(t)$ are given by  
$$I^\alpha_{0|t}u(t)=\frac{1}{\Gamma \left( \alpha  \right)}\int\limits_{0}^{t}{{{\left( t-s \right)}^{\alpha -1}}}u\left( s \right)ds, \,\,\,t\in (0,T]$$
and
$$I^\alpha_{t|T}u(t)=\frac{1}{\Gamma \left( \alpha  \right)}\int\limits_{t}^{T}{{{\left( s-t \right)}^{\alpha -1}}}u\left( s \right)ds, \,\,\,t\in [0,T).$$
\end{definition}
\begin{definition}[\cite{Kilbas}, p. 70] The left and right Riemann-Liouville fractional derivatives $\mathbb{D}_{0|t}^{\alpha }$  of order $\alpha \in(0,1)$, for an absolutely continuous function $u(t)$ is defined by 
$$\mathbb{D}_{0|t}^{\alpha } u(t)=\frac{d}{dt}I_{{0|t}}^{1-\alpha } u(t)=\frac{1}{\Gamma(1-\alpha)}\frac{d}{dt}\int\limits_{0}^{t}{{(t-s)^{-\alpha }}}{u}\left( s \right)ds,\,\,\, \forall t\in (0,T]$$
and
$$\mathbb{D}_{t|T}^{\alpha } u(t)=-\frac{d}{dt}I_{t|T}^{1-\alpha } u(t)=-\frac{1}{\Gamma(1-\alpha)}\frac{d}{dt}\int\limits_{t}^{T}{{(s-t)^{-\alpha }}}{u}\left( s \right)ds,\,\,\, \forall t\in [0,T).$$\end{definition}

\begin{lemma}[\cite{Kilbas}, Lemma 2.20]\label{DI} If $\alpha>0,$ then for $u\in L^1(0,T)$, the relations 
$$\mathbb{D}_{0|t}^{\alpha }I_{0|t}^{\alpha } u(t)=u(t)\,\,\,\text{and}\,\,\,\mathbb{D}_{t|T}^{\alpha }I_{t|T}^{\alpha }u(t)=u(t)$$
are true.

\end{lemma}

\begin{definition}[\cite{Kilbas}, p. 91]\label{CD} The $\alpha \in (0,1)$ order of left and right Caputo fractional derivatives for $u\in C^1([0,T])$ are defined, respectively, by 
$$\mathcal{D}_{0|t}^{\alpha } u(t)=I_{0|t}^{1-\alpha } \frac{d}{dt}u(t)=\frac{1}{\Gamma(1-\alpha)}\int\limits_{0}^{t}{{(t-s)^{-\alpha }}}{u'}\left( s \right)ds,\,\,\, \forall t\in (0,T]$$
and
$$\mathcal{D}_{t|T}^{\alpha } u(t)=-I_{t|T}^{1-\alpha }\frac{d}{dt} u(t)=-\frac{1}{\Gamma(1-\alpha)}\int\limits_{t}^{T}{{(s-t)^{-\alpha }}}{u'}\left( s \right)ds,\,\,\, \forall t\in [0,T).$$
\end{definition}

If $u\in C^1([0,T])$, then the Caputo fractional derivative can be represented by the Riemann-Liouville fractional derivative in the following form
$$\mathcal{D}_{0|t}^{\alpha } u(t)=\mathbb{D}_{0|t}^{\alpha }[u(t)-u(0)],\,\,\, \forall t\in (0,T]$$ and
$$\mathcal{D}_{t|T}^{\alpha } u(t)=\mathbb{D}_{t|T}^{\alpha }[u(t)-u(T)],\,\,\, \forall t\in [0,T).$$

\begin{lemma}[\cite{Zacher}, Corollary 4.1] \label{ZAC} Let $T>0$ and let $U$ be an open subset of $\mathbb{R}$. Let further $u_0\in U,$ $k\in H^1_1(0,T), H\in C^1(U)$ and $u\in L^1(0,T)$ with $u(t)\in U,$ for a. a. $t\in (0,T)$. Suppose that the functions $H(u), H'(u)u$, and $H'(u)(k_t*u)$ belong to $L^1(0,T)$ (which is the case if, e.g., $u\in L^\infty(0,T)$). Assume in addition that $k$ is nonnegative and nonincreasing and that $H$ is convex. Then
$$H'(u(t))\frac{d}{dt}(k*[u-u_0])(t)\geq\frac{d}{dt}(k*[H(u)-H(u_0)])(t),\,\,\,t\in(0,T).$$\end{lemma}

\begin{lemma}[\cite{AA}, Lemma 1]\label{AA} For     
 $0<\alpha<1$ and any function $u(t)$ absolutely continuous and real-valued on $[0,T]$, one has the inequality
\begin{equation*}\begin{split}
\frac{1}{2}\mathcal{D}_{0|t}^{\alpha } (u^2)(t)&\leq u(t)\mathcal{D}_{0|t}^{\alpha } u(t).
\end{split}\end{equation*}
\end{lemma}

\begin{property}[\cite{Kilbas}, p. 95-96]\label{ID}
If $0<\alpha<1$, $u\in AC^1[0,T]$ or $u\in C^1[0,T]$, then
\begin{equation*}\label{I0D}
I^{\alpha}_{0|t}(\mathcal{D}^{\alpha}_{0|t}u)(t)=u(t)-u(0)
\end{equation*}
and
\begin{equation*}\label{I0D1}
\mathcal{D}^{\alpha}_{0|t}(I^{\alpha}_{0|t}u)(t)=u(t),
\end{equation*}
hold almost everywhere on $[0,T]$. In addition,
$$\mathcal{D}^{1-\alpha}_{0|t}\int_0^t\mathcal{D}^{\alpha}_{0|\tau}u(\tau)d\tau=\biggl(I^{\alpha}_{0|t}\frac{d}{dt}I^1_{0|t}I^{1-\alpha}_{0|t}\frac{d}{dt}u\biggr)(t)=u(t)-u(0).$$\end{property}
\begin{property}[\cite{Kilbas}, Lemma 2.7]\label{IBP} Let $0<\alpha<1$ and $u\in C^1[0,T], \varphi\in L^p(0,T)$. Then the integration by parts for Caputo fractional derivatives has the form
\begin{equation*}\label{IBP2}
\int_0^T\left[\mathcal{D}_{0|t}^{\alpha } u\right](t)\varphi(t)dt=\int_0^T u(t)\left[\mathbb{D}_{t|T}^{\alpha }\varphi\right](t)dt+\left[I^{1-\alpha}_{t|T}\varphi\right](t)u(t)\biggr|_0^T.    
\end{equation*}\end{property}

\begin{definition}[\cite{TESO}. 
Lemma 5.1] \label{LP} 
The fractional $p$-Laplacian operator for $s\in(0,1), p>1$ and $u\in W^{s,p}(\Omega)$, is defined by 
\begin{equation*}\label{PLP}
(-\Delta)^s_pu(x)=C_{N,s,p}\,\text{P.V.}\int_{\mathbb{R}^N}\frac{|u(x)-u(y)|^{p-2}(u(x)-u(y))}{|x-y|^{N+sp}}dy,    
\end{equation*}
where 
\begin{equation}\label{CONST}
C_{N,s,p}=\frac{sp2^{2s-2}}{\pi^{\frac{N-1}{2}}}\frac{\Gamma(\frac{N+sp}{2})}{\Gamma(\frac{p+1}{2})\Gamma(1-s)}    
\end{equation} is a normalization constant and “P.V.” is an abbreviation for “in the principal value  sense”. Since they will not play a role in this work, we omit the P.V. sense. However, let us stress that these constants guarantee:
\begin{equation*}\begin{split}
&(-\Delta)^s_pu(x)\xrightarrow[]{s\to1^-}-\Delta_pu(x),\,\,\,\text{for all}\,\,\,\, p\in[2,\infty),
\\& (-\Delta)^s_pu(x)\xrightarrow[]{p\to2^+}(-\Delta)^su(x),\,\,\,\text{for all}\,\,\,\, s\in(0,1).   \end{split}\end{equation*}
\end{definition}

\begin{definition} [\cite{Lindgren}, Theorem 5]\label{FPLD} 
We say that $u\in W^{s,p}_0(\Omega)$ is an $(s,p)$ - eigenfunction associated to the eigenvalue $\lambda$ if $u$ satisfies the Dirichlet problem 
\begin{equation}\label{EF} 
\left\{\begin{array}{l}
(-\Delta)^s_pu(x)=\lambda |u(x)|^{p-2}u(x),\,\,\, x\in\Omega,\\
u(x)=0,\,\,\, x\in \mathbb{R}^N\setminus\Omega,\end{array}\right.\end{equation}
weakly, it means that
 \begin{multline*}
\int_{\Omega}\int_{\Omega}\frac{|u(x)-u(y)|^{p-2}(u(x)-u(y))}{|x-y|^{N+sp}}(\psi(x)-\psi(y))dxdy\\=\lambda \int_{\Omega}|u(x)|^{p-2}u(x)\psi(x)dx,\end{multline*}for every $\psi\in W^{s,p}_0(\Omega)$. If we set as
$$\Sigma_p(\Omega):=\biggl\{u\in W^{s,p}_0(\Omega):\int_{\Omega}|u(x)|^{p}dx=1\biggr\},$$
then the nonlinear Rayleigh quotient determines the first eigenvalue  
\begin{equation*}\label{LAMBADA} 
\lambda_1(\Omega):=\min_{u\in\Sigma_p(\Omega)}\int_{\Omega}\int_{\Omega}\frac{|u(x)-u(y)|^p}{|x-y|^{N+sp}}dxdy.\end{equation*}
\end{definition}

\begin{lemma}[\cite{Lindgren}, Lemma 15]\label{LL} Assume that for all $j$, if we have 
$$\Omega_1\subset\Omega_2\subset\Omega_3\subset...\subset\Omega,\,\,\,\,\Omega=\bigcup \Omega_j.$$ Then
$$\lim_{j\to\infty}\lambda_1(\Omega_j)=\lambda_1(\Omega).$$

Note that the minimization problem is not quite the same if $\mathbb{R}^N\times\mathbb{R}^N$ is replaced by $\Omega\times\Omega$ in the integral. This choice has the advantage that the property
$$\lambda_1(\Omega^*)\leq \lambda_1(\Omega),\,\,\,\text{if}\,\,\,\,\Omega\subset\Omega^*$$
is evident for subdomains. By changing coordinates it implies
$$\lambda_1(\Omega^*)=k^{\alpha p-N}\lambda_1(k\Omega^*),\,\,\,k>0.$$
This asserts that small domains have large first eigenvalues (see \cite{Lindgren} references therein).\end{lemma}

\begin{lemma}[\cite{Brasco1}, Lemma 2.4.  Fractional Poincar\'{e} inequality]\label{Sharp} \\Let $1 \leq p <\infty$ and $s \in (0, 1)$, $\Omega \subset \mathbb{R}^N$ be an open and bounded set. Then, it  holds
\begin{equation*}\label{Po3} \|u\|^p_{L^p(\Omega)}\leq \lambda_1(\Omega)[u]^p_{W^{s,p}(\Omega)}\,\,\text{for}\,\,\,u\in C_0^\infty(\Omega)\end{equation*} and we have the lower bound
$$\lambda_1(\Omega)\geq\frac{1}{\mathcal{I}_{N,s,p(\Omega)}},$$
where the geometric quantity $\mathcal{I}_{N,s,p(\Omega)}$ is defined by 
\end{lemma}
$$\mathcal{I}_{N,s,p(\Omega)}=\min\biggl\{\frac{\text{diam} (\Omega\cup B)^{N+sp}}{|B|},\,\,B\subset\mathbb{R}^N\setminus\Omega\,\,\text{is a ball}\biggr\}.$$
We define the inner product of the operator $(-\Delta)^s_p$ for $u,v\in W^{s,p}(\Omega)$ as
\begin{equation}\label{PSP}
\langle(-\Delta)^s_pu,v\rangle=\int_{\Omega}\int_{\Omega}\frac{|u(x)-u(y)|^{p-2}(u(x)-u(y))(v(x)-v(y))}{|x-y|^{N+sp}}dxdy.
\end{equation}  
\begin{definition}\label{weak}  A function $u=u(x,t)\in W^{s,p}(\Omega;L^\infty(0,T))\cap L^2(\Omega;L^\infty(0,T))$ is called a weak solution of \eqref{01} if the following identity holds 

\begin{equation*}\begin{split}
\int_0^T\!\int\limits_\Omega\mathcal{D}_{0|t}^\alpha u\varphi dxdt&\!+\!\int_0^T\int\limits_{\Omega}\int\limits_{\Omega}\frac{|u(x)\!-\!u(y)|^{p\!-\!2}(u(x)\!-\!u(y))}{|x-y|^{N+sp}}(\varphi(x)\!-\!\varphi(y))dxdydt\\&=\gamma\int_0^T\int_\Omega|u|^{m-1}u\varphi dxdt+\mu\int_0^T\int_\Omega|u|^{q-2}u\varphi dxdt, 
\end{split}\end{equation*} 
almost everywhere in $t\in [0,T]$, for any $\varphi=\varphi(x,t)\in  W_0^{s,p}(\Omega;L^\infty(0,T))$,
such that $\varphi\geq0$ in $\Omega$, $\varphi=0$ on $\partial\Omega$. 
\end{definition}

\subsection{Notations} We recall standard notations, which will be used in the sequel. If $\Omega$ is a bounded and open set in $\mathbb{R}^N$ $(\Omega\subseteq\mathbb{R}^N)$, we denote
$$\Omega_T=\Omega\times(0,T).$$ 
We include the following function space
\begin{equation}\label{Pi}
\Pi=\{u, \mathcal{D}_{0|t}^\alpha u\in W^{s,p}(\Omega;L^\infty(0,T))\cap L^2(\Omega;L^\infty(0,T))\},    
\end{equation}
with the norm
\begin{equation*}\begin{split}
\|u\|^2_\Pi&=\|u\|^2_{W^{s,p}(\Omega;L^\infty(0,T))}+\|u\|^2_{L^2(\Omega;L^\infty(0,T))}\\&+\|\mathcal{D}_{0|t}^\alpha u\|^2_{W^{s,p}(\Omega;L^\infty(0,T))}+\|\mathcal{D}_{0|t}^\alpha u\|^2_{L^2(\Omega;L^\infty(0,T))}.    
\end{split}\end{equation*}

\section{A comparison principle}
\setcounter{section}{3} \setcounter{equation}{0}
In this section we study a comparison principle for the fractional parabolic equation.
We begin by presenting a weak subsolution and a weak supersolution to the problem \eqref{01}. 
\begin{definition}\label{sub}  A real-valued function $$u=u(x,t)\in \Pi, u(x,0)\leq u_0(x), u(x,t)|_{ x\in\partial\Omega}\leq0$$ is called a weak subsolution of \eqref{01} if the inequality  
\begin{equation}\label{WS1}\begin{split}
&\int_0^T\int_{\Omega}\frac{|u(x,t)-u(y,t)|^{p-2}(u(x,t)-u(y,t))}{|x-y|^{N+sp}}(\varphi(x,t)-\varphi(y,t))dxdydt\\&\leq\gamma\int_0^T\int_{\Omega}|u|^{m-1}u\varphi dxdt+\mu\int_0^T\int_{\Omega}|u|^{q-2}u\varphi dxdt\\&-\int_0^T\int_{\Omega}[\mathcal{D}_{0|t}^\alpha u]\varphi dxdt,\end{split}\end{equation} 
holds for any $\varphi\in W_0^{s,p}(\Omega;L^\infty(0,T))$, such that $\varphi\geq0$ in $\Omega$, $\varphi=0$ on $\partial\Omega$. 

Similarly, a real-valued function $$v=v(x,t)\in \Pi, v(x,0)\geq v_0(x), v(x,t)|_{ x\in\partial\Omega}\geq0$$ is called a  weak supersolution of \eqref{01} if it satisfies the inequality
\begin{equation}\label{WS2}\begin{split}
&\int_0^T\int_{\Omega}\frac{|v(x,t)-v(y,t)|^{p-2}(v(x,t)-v(y,t))}{|x-y|^{N+sp}}(\varphi(x,t)-\varphi(y,t))dxdydt\\&\geq\gamma\int_0^T\int_{\Omega}|v|^{m-1}v\varphi dxdt+\mu\int_0^T\int_{\Omega}|v|^{q-2}v\varphi dxdt\\&-\int_0^T\int_{\Omega}[\mathcal{D}_{0|t}^\alpha v]\varphi dxdt.\end{split}\end{equation}

A function is a weak solution, if it is both a weak subsolution and a weak supersolution.\end{definition}

\begin{theorem}\label{CP} Let $s\in(0,1), p\geq 2$ and let $m, q, \gamma, \mu$ satisfy one of the following conditions:
\begin{align*}
&m\geq 1, q\geq 2, \gamma\geq 0, \mu\geq 0;\\&
m>0, q\geq 1, \gamma\leq 0, \mu\leq 0;\\&
m\geq 1, q\geq 1, \gamma\geq 0, \mu\leq 0;\\&
m>0, q\geq 2, \gamma\leq 0, \mu\geq 0.
\end{align*}
Suppose that $u, v\in\Pi$ be real-valued weak subsolution and weak supersolution of \eqref{01}, respectively, with $u_0(x)\leq v_0(x)$ for $x\in \Omega$. Then $u\leq v$ a.e. in  $\Omega_T$.\end{theorem}
\begin{corollary}\label{cor1}
Assume that $p\geq 2$ and let $m, q, \gamma, \mu$ satisfy the conditions in Theorem \ref{CP}. If $u_0(x)\geq 0$ for all $x\in\Omega,$ then $u(x,t)\geq 0,\,x\in\Omega,\,t\geq 0.$
\end{corollary}

The proof of Corollary \ref{cor1} follows from Theorem \ref{CP}. More precisely, if $u_0(x)\geq 0$, taking $0$ as a subsolution, then we have $u(x,t)\geq 0$.
\begin{proof}[Proof of Theorem \ref{CP}] We choose the test function $\varphi=(u-v)_+$, where $(u-v)_+$ is the positive part of a real quantity $(u-v)_+=\max\{u-v, 0\}$. Then it follows that $\varphi(x,0)=0,$  $ \varphi(x,t)|_{\partial\Omega}=0$. By subtracting \eqref{WS2} from \eqref{WS1}, we obtain for $t\in(0,T]$
\begin{equation}\begin{split}\label{MR1}
\int_0^t\int_{\Omega}\mathcal{D}_{0|\tau}^\alpha [u-v]\varphi dxd\tau&+\int_0^t\int_{\Omega}[(-\Delta)^s_pu-(-\Delta)^s_pv]\varphi dxd\tau\\ & \leq\underbrace{\gamma\int_0^t\int_{\Omega}(|u|^{m-1}u-|v|^{m-1}v)\varphi dxd\tau}_{\mathcal{A}}\\&+\underbrace{\mu\int_0^t\int_{\Omega}(|u|^{q-2}u-|v|^{q-2}v)\varphi dxd\tau}_{\mathcal{B}}.\end{split}\end{equation} 
According to \eqref{PSP}, we can write the last term of the left-hand side inequality \eqref{MR1} in the form
\begin{equation}\begin{split}\label{MR2}
&\int_0^t\int_{\Omega}[(-\Delta)^s_pu-(-\Delta)^s_pv]\varphi dxd\tau\\&=\int_0^t\int_{\Omega}\int_{\Omega}\frac{|u(x)-u(y)|^{p-2}(u(x)-u(y))(\varphi(x)-\varphi(y))}{|x-y|^{N+sp}}dxdyd\tau
\\&-\int_0^t\int_{\Omega}\int_{\Omega}\frac{|v(x)-v(y)|^{p-2}(v(x)-v(y))(\varphi(x)-\varphi(y))}{|x-y|^{N+sp}}dxdyd\tau
\\&=\int_0^t\int_{\Omega}\int_{\Omega}\frac{\mathcal{M}(u,v)(\varphi(x)-\varphi(y))}{|x-y|^{N+sp}}dxdyd\tau,\end{split}\end{equation}
where 
\begin{equation}\label{M}
\mathcal{M}(u,v)=|u(x)-u(y)|^{p-2}(u(x)-u(y))-|v(x)-v(y)|^{p-2}(v(x)-v(y)).    
\end{equation}
Hence, we can show that \begin{equation*}\begin{split}
 \mathcal{M}(u,v)&(\varphi(x)-\varphi(y))\\&=\left[|u(x)-u(y)|^{p-2}(u(x)-u(y))-|v(x)-v(y)|^{p-2}(v(x)-v(y))\right]
 \\&\times\left[(u(x)-u(y))-(v(x)-v(y))\right]_+\end{split}
\end{equation*}is nonnegative for any $p\geq2$, thanks to the inequality (see \cite{Lindqvist}, P. 99) 
\begin{equation}\label{I1}
\biggl(\frac{4}{p^2}\biggr)\biggl||a|^{\frac{p-2}{2}}a-|b|^{\frac{p-2}{2}}b\biggr|^2\leq\langle|a|^{p-2}a-|b|^{p-2}b,a-b\rangle,\,\,\text{for}\,\,\,a,b\in\mathbb{R}^N,\end{equation}
with $a:=u(x)-u(y),\, b:=v(x)-v(y)$ in \eqref{M}.  

Now, we will evaluate the right-side of \eqref{MR1}. 
\\Taking account the following inequality 
\begin{equation}\label{I3}
 ||u|^{m-1}u-|v|^{m-1}v|\leq C(m)|u-v|||u|^{m-1}+|v|^{m-1}|\leq L(m)|u-v|,\end{equation} where $L(m)= C(m)\max(\|u\|^{m-1}_{L^\infty(\Omega)},\|v\|^{m-1}_{L^\infty(\Omega)})$, 
we can verify that
\begin{equation}\begin{split}\label{E12}
\mathcal{A}&\leq\gamma C(m)\int_0^t\int_{\Omega}||u|^{m-1}+|v|^{m-1}||u-v|\varphi dxd\tau 
\\&\leq\gamma C(m)\max(\|u\|^{m-1}_{C(\Omega)},\|v\|^{m-1}_{C(\Omega)})\int_0^t\int_{\Omega}|u-v|\varphi dxd\tau
\\&\leq\gamma L(m)\int_0^t\int_{\Omega}|u-v|\varphi dxd\tau
,\end{split}\end{equation} where we have used the well known inequality \cite[Theorem 8.2]{Valdinoci} for any $u\in L^p(\Omega)$ such that
\begin{equation}\label{CCW}
\|u\|_{C(\Omega)}\leq\|u\|_{C^{0,\beta}(\Omega)}\leq\|u\|_{W^{s,p}(\Omega)},\,\,\,\beta=(sp-N)/p,    
\end{equation}
which gives the boundness of $\max(\|u\|^{m-1}_{C(\Omega)},\|v\|^{m-1}_{C(\Omega)})$.

In addition, from the Lipchitsz condition it follows that
\begin{equation}\label{I4}
||u|^{q-2}u-|v|^{q-2}v|\leq L(q)|u-v|,\,\,\,q\geq1,\end{equation}
where $L(q)= C(q)\max(\|u\|^{q-2}_{L^\infty(\Omega)},\|v\|^{q-2}_{L^\infty(\Omega)})$. Hence, from \eqref{CCW}, it follows that
\begin{equation}\begin{split}\label{E13}
\mathcal{B}&\leq\mu C(q)\int_0^t\int_{\Omega}||u|^{q-2}+|v|^{q-2}||u-v|\varphi dxd\tau
\\&\leq\mu C(q)\max(\|u\|^{q-2}_{C(\Omega)},\|v\|^{q-2}_{C(\Omega)}) \int_0^t\int_{\Omega}|u-v|\varphi dxd\tau
\\&\leq \mu L(q)\int_0^t\int_{\Omega}|u-v|\varphi dxd\tau.\end{split}\end{equation}
Combining \eqref{MR2}, \eqref{E12} and \eqref{E13}, we can rewrite the inequality \eqref{MR1} as 
\begin{equation}\begin{split}\label{MR3}
&\int_0^t\int_{\Omega}(\mathcal{D}_{0|\tau}^\alpha [u-v])(u-v)_+ dxd\tau\\&\leq \left(\gamma L(m)+\mu L(q)\right)\int_0^t\int_{\Omega}|u-v|(u-v)_+ dxd\tau.\end{split}\end{equation} 
Using Lemma \ref{AA}, the inequality \eqref{MR3} can be rewritten in the following form
\begin{equation}\begin{split}\label{MJR3}
\frac{1}{2}\int_0^t\int_{\Omega}\mathcal{D}_{0|\tau}^\alpha (u-v)_+^2dxd\tau&\leq \left(\gamma L(m)+\mu L(q)\right)\int_0^t\int_{\Omega}(u-v)_+^2dxd\tau.\end{split}\end{equation} 

At this stage, we have to consider three cases depending on $\gamma, \mu$:

$\bullet$ The case $\gamma\geq 0,\mu\geq0.$ 
Applying the left Caputo fractional differentiation operator $\mathcal{D}_{0|t}^{1-\alpha}$ to both sides of \eqref{MJR3} and using Property \ref{ID}, we obtain
\begin{equation}\begin{split}\label{MRM7}
\frac{1}{2}\int_{\Omega}(u-v)_+^2dx\leq \left(\gamma L(m)+\mu L(q)\right)\int_{\Omega}\int_0^t(t-\tau)^{\alpha-1}(u-v)_+^2d\tau dx.\end{split}\end{equation} 
Then, from the weakly singular Gronwall's inequality (see \cite{Henry}, Lemma 7.1.1 and \cite{Haraux}, Lemma 6, p. 33) $$\int_{\Omega}(u-v)_+^2dx=0 \iff (u-v)_+=0,\,\,\,x\in\Omega.$$
Finally, it follows that $u\leq v$ almost everywhere for $(x,t)\in \Omega_T$.

$\bullet$ The case  $\gamma\leq0,\mu\leq0.$ 
According to the inequality \eqref{MJR3}, the right-hand side integral is positive and the coefficients $\gamma,\mu$ are non-positive, we deduce that
\begin{equation*}\begin{split}
\frac{1}{2}\int_0^t\int_{\Omega}\mathcal{D}_{0|\tau}^\alpha (u-v)_+^2 dxd\tau&\leq \left(\gamma L(m)+\mu L(q)\right)\int_0^t\int_{\Omega}(u-v)_+^2dxd\tau\\&\leq 0.\end{split}\end{equation*} 
Therefore, repeating the similar procedure as above we obtain
\begin{equation*}\begin{split}
\int_{\Omega}(u-v)_+^2 dx=0.\end{split}\end{equation*}
Consequently, we have $u\leq v$ almost everywhere for $(x,t)\in \Omega_T$.

$\bullet$ The case  $\gamma\geq0,\mu\leq0$ or $\gamma\leq0,\mu\geq0$. Using the inequality \eqref{MJR3}, it follows that 
\begin{equation*}\begin{split}
\frac{1}{2}\int_{\Omega}(u-v)_+^2dx\leq \gamma L(m)\int_{\Omega}\int_0^t(t-\tau)^{\alpha-1}(u-v)_+^2d\tau dx\end{split}\end{equation*}
or
\begin{equation*}\begin{split}
\frac{1}{2}\int_{\Omega}(u-v)_+^2 dx\leq \mu L(q) \int_{\Omega}\int_0^t(t-\tau)^{\alpha-1}(u-v)_+^2d\tau dx,\end{split}\end{equation*} respectively. By the weakly singular Gronwall's inequality, we arrive at $u\leq v$ almost everywhere for $(x,t)\in \Omega_T$.  \end{proof}

\section{Local well-posedness}
\setcounter{section}{4} \setcounter{equation}{0}
\subsection{Existence of a local weak solution}
In this subsection, we will prove that problem \eqref{01} has the local weak solution by Galerkin method.

\begin{theorem}\label{thloc} Let $u_{0} \in W^{s,p}_0(\Omega), u_0\geq0,\, sp<N$ and let either $1<m<q-1<p-1$ or $1<q-1<m <p-1.$ Then there exists $T>0$ such that the problem \eqref{01} has a local real-valued weak solution $u\in \Pi$, where $\Pi$ is defined in \eqref{Pi}.\end{theorem}
\proof $\bullet$ The case $1<m<q-1<p-1$. The space $W^{s,p}_0(\Omega)$ is separable. Then there exists a countable linear set $\{\omega_j\}_{j\in N}$ that is everywhere dense in $W^{s,p}_0(\Omega)$.   

Let us consider the Galerkin approximations
\begin{equation}\label{G1}u_n(x,t)=\sum_{j=1}^n v_{nj}(t)\omega_j(x),\end{equation}
where the unknown $v_{nj}\in C^1([0,T_n])$ functions  satisfy the following system of ordinary fractional differential equations:
\begin{equation}\label{G2}\begin{split}
&\int_\Omega\mathcal{D}_{0|t}^{\alpha }u_{n}\omega_k dx+P(u_{n}, \omega_k)
\\&=\gamma\int_\Omega|u_n|^{m-1}u_n\omega_kdx+\mu\int_\Omega|u_n|^{q-2}u_n\omega_kdx,\,\,k=1,2,...,n,   \end{split}\end{equation} supplemented by the initial condition 
\begin{equation}\label{CD0}\begin{split}
u_n(x,0)=\sum_{j=1}^n v_{nj}(0)\omega_j\xrightarrow[]{n\to\infty}u_0\,\,\text{in}\,\,W_0^{s,p}(\Omega),\end{split} \end{equation}
where
\begin{equation*}\begin{split}
P(u_{n},\omega_k)&
\!=\!\int_{\Omega}\int_{\Omega}\biggl|u_{n}(x,t)\!-\!u_{n}(y,t)\biggr|^{p\!-\!2}\!(u_{n}(x,t)\!-\!u_{n}(y,t))\frac{\omega_{k}(x)\!-\!\omega_{k}(y)}{|x-y|^{N+sp}}dxdy.\end{split}\end{equation*}
First of all, we need to prove that the system
of Galerkin equations \eqref{G2} has a solution $v_{nj}\in C^1([0, T_n]),\,j=\overline{1,n}$ for some $T_n > 0$,
which depends on $n\in N$.
Therefore, we note that the system of equations \eqref{G2} can be represented in the following form
\begin{equation}\label{GA1}\begin{split}
\sum_{j=1}^n a_{jk}\mathcal{D}_{0|t}^{\alpha }v_{nj}(t)+F_{1k}(v_n)=F_{2k}(v_n)+F_{3k}(v_n),\,\,\,\end{split} \end{equation}
where $a_{jk}$ is an invertible matrix for each $n\in N$ and the functions $F_{ik}(v_n),\,i=1,2,3$ are defined by  
\begin{equation*}\label{GA2}\begin{split}
F_{1k}(v_n)&\!=\!\int_{\Omega}\int_{\Omega}\sum_{j=1}^n\biggl|v_{nj}(t)\omega_{j}(x)\!-\!v_{nj}(t)\omega_{j}(y)\biggr|^{p-2}(v_{nj}(t)\omega_{j}(x)-v_{nj}(t)\omega_{j}(y))\\&\times (\omega_{k}(x)-\omega_{k}(y)) \frac{1}{|x-y|^{N+sp}}dxdy\,\,\,\end{split} \end{equation*}
and
\begin{equation*}\label{GA3}\begin{split}
F_{2k}(v_n)=\int_\Omega\sum_{j=1}^n|v_{nj}(t)\omega_j(x)|^{m-1}v_{nj}(t)\omega_j(x)\omega_k(x)dx,\,\,\,\end{split} \end{equation*}
\begin{equation*}\label{GA4}\begin{split}
F_{3k}(v_n)=\int_\Omega\sum_{j=1}^n|v_{nj}(t)\omega_j(x)|^{q-2}v_{nj}(t)\omega_j(x)\omega_k(x)dx.\,\,\,\end{split} \end{equation*}
Next, we will prove the functions $F_{ik}(v_n),\,i=1,2,3$ are locally Lipschitz function. Indeed, we have
\begin{equation*}\label{GA5}\begin{split}
&|F_{1k}(v^1_n)-F_{1k}(v^2_n)|\\&=\int_{\Omega}\int_{\Omega}\sum_{j=1}^n\biggl(\biggl|v^1_{nj}(t)\omega_{j}(x)-v^1_{nj}(t)\omega_{j}(y)\biggr|^{p-2}(v^1_{nj}(t)\omega_{j}(x)-v^1_{nj}(t)\omega_{j}(y))\\&\!-\!\biggl|v^2_{nj}(t)\omega_{j}(x)\!-\!v^2_{nj}(t)\omega_{j}(y)\biggr|^{p\!-\!2}\!(v^2_{nj}(t)\omega_{j}(x)\!-\!v^2_{nj}(t)\omega_{j}(y))\biggr)\!\frac{\omega_{k}(x)\!-\!\omega_{k}(y)}{|x\!-\!y|^{N+sp}}\!dxdy.\end{split} \end{equation*}
Using the following inequality, for $p\geq1,$
\begin{equation*}\begin{split}
||\varphi|^{p-2}\varphi-|\psi|^{p-2}\psi|&\leq C(p)||\varphi|^{p-2}+|\psi|^{p-2}||\varphi-\psi|\\&\leq C(p)\max\{|\varphi|^{p-2};|\psi|^{p-2}\}|\varphi-\psi|  
\end{split}
\end{equation*}
and the generalized H\"{o}lder inequality with parameters
\begin{equation}\label{GHI}
\frac{1}{r_1}+\frac{1}{r_2}+\frac{1}{r_3}=1,\,\,\,r_1=\frac{p}{p-2},\,\,\,r_2=p,\,\,\,r_3=p,\end{equation}
in the last equality, we obtain 
\begin{equation*}\label{GA6}\begin{split}
&|F_{1k}(v^1_n)-F_{1k}(v^2_n)|\\&\leq\int_{\Omega}\int_{\Omega}\frac{\max\biggl\{|u^1_{n}(x,t)-u^1_{n}(y,t)|^{p-2};|u^2_{n}(x,t)-u^2_{n}(y,t)|^{p-2}\biggr\}}{|x-y|^{\frac{(N+sp)(p-2)}{p}}}
\\&\times\sum_{j=1}^n\frac{|(v^1_{nj}(t)\omega_{j}(x)-v^1_{nj}(t)\omega_{j}(y))-(v^2_{nj}(t)\omega_{j}(x)-v^2_{nj}(t)\omega_{j}(y))|}{|x-y|^{\frac{N+sp}{p}}}
\\&\times\frac{\omega_{k}(x)-\omega_{k}(y)}{|x-y|^{{\frac{N+sp}{p}}}}dxdy
\\&\leq C(p)\Phi^{p-2}[\omega_k]_{W^{s,p}(\Omega)}[u^1_n-u_n^2]_{W^{s,p}(\Omega)},
\end{split} \end{equation*} where $\Phi^{p-2}=\max\{[u_n^1]_{W^{s,p}(\Omega)}; [u_n^2]_{W^{s,p}(\Omega)}\}$ and $[\,\cdot\,]_{W^{s,p}(\Omega)}$ is the Gagliardo semi-norm.
Consequently,
\begin{equation*}\label{GA7}\begin{split}
&[u^1_n-u_n^2]_{W^{s,p}(\Omega)}
\\&=\int_{\Omega}\int_{\Omega}\sum_{j=1}^n\frac{|v^1_{nj}(t)\omega_{j}(x)-v^1_{nj}(t)\omega_{j}(y))-(v^2_{nj}(t)\omega_{j}(x)-v^2_{nj}(t)\omega_{j}(y)|^p}{|x-y|^{N+sp}}dxdy
\\&=\int_{\Omega}\int_{\Omega}\sum_{j=1}^n\frac{|v^1_{nj}(t)-v^2_{nj}(t)|^p|\omega_{j}(x)-\omega_{j}(y)|^p}{|x-y|^{N+sp}}dxdy
\\&\leq [w_j]_{W^{s,p}(\Omega)}|v^1_{n}-v^2_{n}|^p.
\end{split} \end{equation*} 

At this stage using 
$$|a-b|^p\leq|a-b||a-b|^{p-1}\leq2^{p-2}|a+b|^{p-1}|a-b|,\,\,\,p>2, a,b\in \mathbb{R},$$
in the last term of the previous inequality, and recalling $v^1_{nj}, v^2_{nj}\in C^1([0, T_n])$ we arrive at
\begin{equation*}\label{GA71}\begin{split}
&|F_{1k}(v^1_n)-F_{1k}(v^2_n)|\leq C(p)\Phi^{p-2}[\omega_k]_{W^{s,p}(\Omega)}\max\{|v^1_n|^{p-1};|v_n^2|^{p-1}\}|v^1_n-v_n^2|.
\end{split} \end{equation*}
Accordingly, using the inequalities \eqref{I3} and \eqref{CCW} to $F_{2k}(v_n)$, for $k,j=\overline{1,n}$, we deduce that
\begin{equation*}\label{GA8}\begin{split}
&|F_{2k}(v^1_n)-F_{2k}(v^2_n)|\\& \leq\int_\Omega\sum_{j=1}^n||v^1_{nj}(t)\omega_j(x)|^{m\!-\!1}v^1_{nj}(t)\omega_j(x)\!-\!|v^2_{nj}(t)\omega_j(x)|^{m\!-\!1}v^2_{nj}(t)\omega_j(x)||\omega_k(x)|dx
\\&\leq\max\{\|v^1_nw_j\|_{C(\Omega)}^{m-1};\|v_n^2w_j\|_{C(\Omega)}^{m-1}\}|v^1_n-v^2_n|\int_\Omega|\omega_j(x)||\omega_k(x)|dx
\\&\leq\max\{\|v^1_nw_j\|_{C(\Omega)}^{m-1};\|v_n^2w_j\|_{C(\Omega)}^{m-1}\}\|w_j\|_{L^2(\Omega)}\|w_k\|_{L^2(\Omega)}|v^1_n-v^2_n|.\,\,\,\end{split} \end{equation*}
Similarly, from \eqref{CCW} and \eqref{I4}  we obtain an estimate for $F_{3k}(v_n)$, for $k,j=\overline{1,n}$, in the following form
\begin{equation*}\label{GA9}\begin{split}
&|F_{3k}(v^1_n)-F_{3k}(v^2_n)|\\&\leq\int_\Omega\sum_{j=1}^n||v^1_{nj}(t)\omega_j(x)|^{q-2}v^1_{nj}(t)\omega_j(x)\!-\!|v^2_{nj}(t)\omega_j(x)|^{q-2}v^2_{nj}(t)\omega_j(x)||\omega_k(x)|dx
\\&\leq \max\{\|v^1_nw_j\|_{C(\Omega)}^{q-2};\|v_n^2w_j\|_{C(\Omega)}^{q-2}\}|v^1_n-v^2_n|\int_\Omega|\omega_j(x)||\omega_k(x)|dx
\\&\leq \max\{\|v^1_nw_j\|_{C(\Omega)}^{q-2};\|v_n^2w_j\|_{C(\Omega)}^{q-2}\}\|w_j\|_{L^2(\Omega)}\|w_k\|_{L^2(\Omega)}|v^1_n-v^2_n|.\,\,\,\end{split} \end{equation*}
From Lemma \ref{EMBD}, the space $W^{s,p}(\Omega)$ is continuously embedded in $L^2(\Omega)$. Indeed, the right-hand side  of $F_{ik}(v_n),\,\,i=1,2,3,\,k=\overline{1,n}$ is continuous with respect to $t\in [0, T_n]$ and locally  Lipschitz function with respect to $v_n(t)$. 

Therefore, due to \cite[Theorem 3.25]{Kilbas} the Cauchy problem for the system of equations \eqref{GA1} has a unique solution $v_{nj}\in C^1([0, T_n]),\,j=\overline{1,n}$ for some $T_n > 0$, which depends on $n\in N$. 

Multiplying the expression \eqref{G2} by $v_{nk}(t)$ and performing the summation over $k=1,...,n,$ it follows that
\begin{equation*}\label{G3}\begin{split}
&\int_\Omega u_{n}\mathcal{D}_{0|t}^{\alpha }u_{n}dx+[u_n]^p_{W^{s,p}(\Omega)}=\gamma\int_\Omega|u_n|^{m+1}dx+\mu\int_\Omega|u_n|^{q}dx.\end{split}\end{equation*}
Applying the fractional Poincar\'{e} inequality from Lemma \ref{Sharp} and the inequality in Lemma \ref{AA} to the previous identity, we get 
\begin{equation}\label{G4}\begin{split}
\frac{1}{2}\mathcal{D}_{0|t}^{\alpha }\int_\Omega |u_{n}|^2dx&+\frac{1}{\lambda_1(\Omega)}\int_\Omega|u_n|^{p}dx
\\&\leq\gamma\int_\Omega|u_n|^{m+1}dx+\mu\int_\Omega|u_n|^{q}dx.\end{split}\end{equation}

At this stage we have to consider different cases of coefficients $\gamma$ and $\mu$.
\newline$\bullet$ The case $\gamma,\,\mu>0$. Thanks to the inequality (see  \cite{Soup}, P. 417), 
\begin{equation}\label{INQ1} z^{b+c-1}\leq\varepsilon z^c+C(a,b)\varepsilon^{-\frac{a-b}{b-1}}z^{a+c-1},\,\,a>b,c>1,\,\,\,\text{and}\,\,\, z\geq0,\,\,\varepsilon>0,\end{equation}
for $a=q-1,\,b=m,$ and $c=2$ in \eqref{G4} we obtain
\begin{equation}\label{INQ3} \gamma\int_\Omega |u_{n}|^{m+1}dx\leq\gamma\varepsilon \int_\Omega |u_{n}|^2dx+\gamma C(q,m)\varepsilon^{-\frac{q-1-m}{m-1}}\int_\Omega |u_{n}|^qdx.\end{equation}
Consequently, it follows that
\begin{equation}\label{GO5}\begin{split}
\mathcal{D}_{0|t}^{\alpha }\int_\Omega |u_{n}|^2dx&\leq\gamma\varepsilon \int_\Omega |u_{n}|^2dx-\frac{1}{\lambda_1(\Omega)}\int_\Omega|u_n|^{p}dx\\&+\biggl(\gamma C(q,m)\varepsilon^{-\frac{q-1-m}{m-1}}+\mu\biggr)\int_\Omega|u_n|^{q}dx.\end{split}\end{equation}
Due to the inequality \eqref{INQ1} for $a=p-1,\,b=q-1$ and $c=2$, it holds  
\begin{equation*}\label{INQO4}\int_\Omega |u_{n}|^qdx\leq\tilde{\varepsilon} \int_\Omega |u_{n}|^2dx+C(p,q)\tilde{\varepsilon}^{-\frac{p-q}{q-1}}\int_\Omega |u_{n}|^pdx. \end{equation*}
Therefore, using the last inequality in  \eqref{GO5} we get  
\begin{equation*}\label{GO6}\begin{split}
\mathcal{D}_{0|t}^{\alpha }\int_\Omega |u_{n}|^2dx&\leq\biggl(\tilde{\varepsilon}\gamma C(q,m)\varepsilon^{-\frac{q-1-m}{m-1}}+\tilde{\varepsilon}\mu+\gamma\varepsilon\biggr) \int_\Omega |u_{n}|^2dx\\&+\biggl[\biggl(\gamma C(q,m)\varepsilon^{-\frac{q-1-m}{m-1}}+\mu\biggr)C(p,q)\tilde{\varepsilon}^{-\frac{p-q}{q-1}}-\frac{1}{\lambda_1(\Omega)}\biggr]\int_\Omega|u_n|^{p}dx.\end{split}\end{equation*}
Finally, choosing the constants $\varepsilon, \tilde{\varepsilon}>0$ such that
$$\biggl(\gamma C(q,m)\varepsilon^{-\frac{q-1-m}{m-1}}+\mu\biggr)C(p,q)\tilde{\varepsilon}^{-\frac{p-q}{q-1}}-\frac{1}{\lambda_1(\Omega)}\leq0,$$ then we get the following result
\begin{equation}\label{G6}\begin{split}
\mathcal{D}_{0|t}^{\alpha }\int_\Omega |u_{n}|^2dx\leq C(\varepsilon,\tilde{\varepsilon}) \int_\Omega |u_{n}|^2dx,\end{split}\end{equation}where $$ C(\varepsilon,\tilde{\varepsilon})=\tilde{\varepsilon}\gamma C(q,m)\varepsilon^{-\frac{q-1-m}{m-1}}+\tilde{\varepsilon}\mu+\gamma\varepsilon.$$
Define $\Large\displaystyle\Phi(t):=\int_\Omega |u_{n}|^2dx$, then applying the left Riemann-Liouville fractional integral operator $I_{0|t}^{\alpha}$ to both sides of \eqref{G6} and using Property \ref{ID}, we get
\begin{equation*}\begin{split}\label{G7}
\Phi(t)\leq \Phi(0)+C(\varepsilon,\tilde{\varepsilon})\int_0^t(t-s)^{\alpha-1}\Phi(s)ds .\end{split}\end{equation*} 
Furthermore, according to Gronwall-type inequality for fractional integral equations (see \cite{Diethelm}, Lemma 4.3) we obtain  
$$\Phi(t)\leq\Phi(0)E_{\alpha,1}(C(\varepsilon,\tilde{\varepsilon}) t^\alpha)\,\,\,\text{for all}\,\,\,t\in[0,T],$$
where $E_{\alpha,1 }(z)$ is the Mittag-Leffler function, defined by
$$E_{\alpha,1 }(z)=\sum_{k=0}^{\infty}\frac{z^k}{\Gamma(\alpha k+1)},\,z\geq0.$$
Finally, in view of Corollary \ref{cor1} for real-valued $u$, we conclude that there exists finite $T_0>0$,
\begin{equation}\label{L2}
\|u_{n}(\cdot,t)\|^2_{L^2(\Omega)}\leq \|u_{n}(\cdot,0)\|^2_{L^2(\Omega)}E_{\alpha,1}(C(\gamma,\varepsilon) t^\alpha)=A(T),\end{equation}for all $t\in[0,T]$,\,$T<T_0$, where $A(T)$ is a constant independent of $n$.

\hfill \break $\bullet$ The case $\gamma>0$ and $\mu\leq0$. Then from \eqref{G4} we obtain
\begin{equation*}\begin{split}
\frac{1}{2}\mathcal{D}_{0|t}^{\alpha }\int_\Omega |u_{n}|^2dx&+\frac{1}{\lambda_1(\Omega)}\int_\Omega|u_n|^{p}dx\leq\gamma\int_\Omega|u_n|^{m+1}dx.\end{split}\end{equation*}
Setting $a=p-1,\,b=m,$ and $c=2$ in \eqref{INQ3} we can rewrite the last estimate as
\begin{equation*}\begin{split}
\frac{1}{2}\mathcal{D}_{0|t}^{\alpha }\int_\Omega |u_{n}|^2dx&\leq\gamma\varepsilon \int_\Omega |u_{n}|^2dx\!+\!\biggl(\gamma C(p,m)\varepsilon^{-\frac{p-1-m}{m-1}}\!-\!\frac{1}{\lambda_1(\Omega)}\biggr)\int_\Omega |u_{n}|^pdx.\end{split}\end{equation*}
By choosing the constants $\varepsilon, \tilde{\varepsilon}>0$ which satisfy
$$\gamma C(p,m)\varepsilon^{-\frac{p-1-m}{m-1}}-\frac{1}{\lambda_1(\Omega)}\leq0,$$ then we get
\begin{equation*}\begin{split}
\mathcal{D}_{0|t}^{\alpha }\int_\Omega |u_{n}|^2dx\leq \gamma\varepsilon\int_\Omega |u_{n}|^2dx.\end{split}\end{equation*}
The conclusion can be derived as in the previous case.

\hfill \break $\bullet$ The case $\gamma\leq0$ and $\mu>0$. Accordingly from \eqref{G4} we have
\begin{equation*}\begin{split}
\frac{1}{2}\mathcal{D}_{0|t}^{\alpha }\int_\Omega |u_{n}|^2dx&+\frac{1}{\lambda_1(\Omega)}\int_\Omega|u_n|^{p}dx\leq\mu\int_\Omega|u_n|^{q}dx.\end{split}\end{equation*}
Next, choosing $a=p-1,\,b=q-1,$ and $c=2$ in \eqref{INQ3} it follows
\begin{equation*}\begin{split}
\frac{1}{2}\mathcal{D}_{0|t}^{\alpha }\int_\Omega |u_{n}|^2dx&\leq\mu\varepsilon \int_\Omega |u_{n}|^2dx+\biggl(\mu C(p,q)\varepsilon^{-\frac{p-q}{q-2}}-\frac{1}{\lambda_1(\Omega)}\biggr)\int_\Omega |u_{n}|^pdx.\end{split}\end{equation*}
Now, taking $\varepsilon, \tilde{\varepsilon}>0$, which satisfy
$$\mu C(p,q)\varepsilon^{-\frac{p-q}{q-2}}-\frac{1}{\lambda_1(\Omega)}\leq0,$$ we obtain the estimate
\begin{equation*}\begin{split}
\mathcal{D}_{0|t}^{\alpha }\int_\Omega |u_{n}|^2dx\leq \mu\varepsilon\int_\Omega |u_{n}|^2dx.\end{split}\end{equation*}
Similarly, the conclusion can be derived as in the previous case.

\hfill \break $\bullet$ The case $\gamma,\mu\leq0$. Take into consideration the inequality \eqref{G4} it yields
\begin{equation*}\begin{split}
\frac{1}{2}\mathcal{D}_{0|t}^{\alpha }\int_\Omega |u_{n}|^2dx&+\frac{1}{\lambda_1(\Omega)}\int_\Omega|u_n|^{p}dx\leq0.\end{split}\end{equation*}
Using the fact that $\lambda_1$ is nonnegative we obtain
\begin{equation*}\begin{split}
\frac{1}{2}\mathcal{D}_{0|t}^{\alpha }\int_\Omega |u_{n}|^2dx\leq0.\end{split}\end{equation*}
Hence, applying the left Riemann-Liouville integral $I_{0|t}^{\alpha}$ to the last inequality and using Property \ref{ID}, we deduce that
\begin{equation*}\begin{split}\int_\Omega |u_{n}(x,t)|^2dx\leq\int_\Omega |u_{n}(x,0)|^2dx.\end{split}\end{equation*}
Finally, it follows that
$$\|u_n(\cdot,t)\|_{L^2(\Omega)}\leq\|u_{n}(x,0)\|_{L^2(\Omega)},\,\,\,\text{for all}\,\,\,t\geq0.$$
Next, multiplying the expression \eqref{G2} by $\mathcal{D}_{0|t}^{\alpha }v_{nk}(t)$ and summing over $k=\overline{1,n}$, we obtain
\begin{equation}\label{G8}\begin{split}
\|\mathcal{D}_{0|t}^{\alpha }u_{n}\|^2_{L^2(\Omega)}&+ P(u_{n},\mathcal{D}_{0|t}^{\alpha }u_{n}(t))
\\&=\gamma\int_\Omega|u_n|^{m-1}u_n\mathcal{D}_{0|t}^{\alpha }u_{n}dx+\mu\int_\Omega|u_n|^{q-2}u_n\mathcal{D}_{0|t}^{\alpha }u_{n}dx,   
\end{split}\end{equation}
with
\begin{equation}\begin{split}\label{Z1}
P(u_{n},\mathcal{D}_{0|t}^{\alpha }u_{n}(t))&
=\int_{\Omega}\int_{\Omega}\frac{|u_{n}(x,t)-u_{n}(y,t)|^{p-2}}{|x-y|^{N+sp}}(u_{n}(x,t)-u_{n}(y,t))\\&\times\mathcal{D}_{0|t}^{\alpha }[u_{n}(x,t)-u_{n}(y,t)]dxdy.  \end{split}\end{equation}
Due to Lemma \ref{AA} it follows that
$$(u_{n}(x,t)-u_{n}(y,t))\mathcal{D}_{0|t}^{\alpha }[u_{n}(x,t)-u_{n}(y,t)]\geq\frac{1}{2}\mathcal{D}_{0|t}^{\alpha }[u_{n}(x,t)-u_{n}(y,t)]^2 .$$
Moreover the identity \eqref{Z1} becomes
\begin{equation}\begin{split}\label{Z11}
P(u_{n},\mathcal{D}_{0|t}^{\alpha }u_{n}(t))&
\geq\frac{1}{2}\int_{\Omega}\int_{\Omega}\frac{|u_{n}(x,t)-u_{n}(y,t)|^{p-2}}{|x-y|^{N+sp}}\\&\times \mathcal{D}_{0|t}^{\alpha }[u_{n}(x,t)-u_{n}(y,t)]^2dxdy.  \end{split}\end{equation}
\\At this stage, we consider the function $$H(\omega)(t)=\frac{2}{p}|\omega(t)|^\frac{p}{2},\,\,p\geq2,$$ which is convex. By differentiating respect to $\omega$ we have $H'(\omega)(t)=|\omega(t)|^\frac{p-2}{2}$. From Lemma \ref{ZAC} for the function $H(\omega)(t)$ we obtain  the following inequality
$$|\omega(t)|^\frac{p-2}{2}\mathcal{D}^\alpha_{0|t}\omega(t)\geq \frac{2}{p}\mathcal{D}^\alpha_{0|t}|\omega|^\frac{p}{2}(t).$$
Denote $\omega(t)=|u_{n}(x,t)-u_{n}(y,t)|^2$. Then, we obtain
\begin{equation*}\label{G081}\begin{split}
|u_{n}(x,t)-u_{n}(y,t)|^{p-2}\mathcal{D}^\alpha_{0|t}|u_{n}(x,t)-u_{n}(y,t)|^2\geq\frac{1}{p} \mathcal{D}^\alpha_{0|t}|u_{n}(x)-u_{n}(y)|^{p}.
\end{split}\end{equation*} 
Therefore, using \eqref{Z11} and the last inequality we get
\begin{equation*}\begin{split}
\left|P(u_{n},\mathcal{D}_{0|t}^{\alpha }u_{n}(t))\right|\geq
&\frac{1}{p}\int_{\Omega}\int_{\Omega}\frac{1}{|x-y|^{N+sp}} \mathcal{D}_{0|t}^{\alpha }|u_{n}(x,t)-u_{n}(y,t)|^{p}dxdy. 
\end{split}\end{equation*} 
Since the operator $\mathcal{D}_{0|t}^{\alpha }$ is with respect to the variable $t$ it follows that
 \begin{equation*}\begin{split}
\left|P(u_{n},\mathcal{D}_{0|t}^{\alpha }u_{n}(t))\right|&\geq
\frac{1}{p}\mathcal{D}_{0|t}^{\alpha }\int_{\Omega}\int_{\Omega}\frac{|u_{n}(x,t)-u_{n}(y,t)|^{p}}{|x-y|^{N+sp}} dxdy
\\&=\frac{1}{p}\mathcal{D}_{0|t}^{\alpha }[u_{n}(\cdot, t)]^p_{W^{s,p}(\Omega)}. 
\end{split}\end{equation*} 
 Finally, the identity \eqref{G8} can be rewritten as
 \begin{equation}\label{G80}\begin{split}
\|\mathcal{D}_{0|t}^{\alpha }u_{n}\|^2_{L^2(\Omega)}&+ \frac{1}{p}\mathcal{D}_{0|t}^{\alpha }[u_{n}(\cdot, t)]^p_{W^{s,p}(\Omega)}
\\&\leq\gamma\int_\Omega|u_n|^{m-1}u_n\mathcal{D}_{0|t}^{\alpha }u_{n}dx+\mu\int_\Omega|u_n|^{q-2}u_n\mathcal{D}_{0|t}^{\alpha }u_{n}dx.   
\end{split}\end{equation}
At this stage, we should study the different cases of the coefficients $\gamma$ and $\mu$.
\newline $\bullet$ The case $\gamma,\mu>0$. Using the H\"{o}lder and $\varepsilon$-Young inequalities
$$XY\leq \frac{\varepsilon}{p} X^p+C(\varepsilon)Y^{p'},\,\, \frac{1}{p}+\frac{1}{p'}=1,\,\, X,Y\geq0,$$where $\large\displaystyle C(\varepsilon)=\frac{1}{p'\varepsilon^{p'-1}}$ for the right hand side of \eqref{G80}, respectively, we get
\begin{equation}\label{G9}\begin{split}
\gamma\int _{\Omega}|u_{n}|^{m-1}&u_{n}\mathcal{D}_{0|t}^{\alpha}u_{n}dx \leq\gamma\left(\int _{\Omega}|u_{n}|^{2m}dx\right)^\frac{1}{2}\left(\int _{\Omega}\left|\mathcal{D}_{0|t}^{\alpha}u_{n}\right|^2dx\right)^\frac{1}{2}
\\&\le  \gamma \left\| u_{n} (\cdot, t) \right\|^m _{L^{2m}(\Omega) } \left\| \mathcal{D}_{0|t}^{\alpha } u_{n}(\cdot, t) \right\| _{L^2(\Omega) }
\\&\leq \frac{\varepsilon}{2}\gamma^2\left\|u_{n}(\cdot, t)\right\|_{L^{2m}(\Omega)}^{2m}+\frac{1}{2\varepsilon} \left\|\mathcal{D}_{0|t}^{\alpha}u_{n}(\cdot, t)\right\|_{L^2(\Omega)}^{2}
\end{split}\end{equation}
and
\begin{equation}\label{G10}\begin{split}
\mu\int_\Omega|u_n|^{q-2}&u_n\mathcal{D}_{0|t}^{\alpha }u_{n}dx\leq \mu \left(\int _{\Omega }|u_{n} |^{2(q-1)} dx \right)^{\frac{1}{2} } \left(\int _{\Omega }|\mathcal{D}_{0|t}^{\alpha } u_{n} |^{2} dx \right)^{\frac{1}{2} } 
\\&\le  \mu \left\| u_{n} (\cdot, t) \right\|^{q-1}_{L^{2(q-1)}(\Omega) } \left\| \mathcal{D}_{0|t}^{\alpha } u_{n}(\cdot, t) \right\| _{L^2(\Omega) } 
\\& \le  \frac{\varepsilon_1}{2} \mu^2\left\| u_{n} (\cdot, t) \right\|^{2(q-1)} _{L^{2(q-1)}(\Omega) }+\frac{1}{2\varepsilon_1} \left\| \mathcal{D}_{0|t}^{\alpha } u_{n}(\cdot, t) \right\| _{L^2(\Omega) }^{2}.   
\end{split}\end{equation}
From Lemma \ref{FCKN} we obtain 
\begin{equation}\begin{split}\label{FC1}
\frac{\varepsilon}{2}\gamma^2&\left\|u_{n}(\cdot, t)\right\|_{L^{2m}(\Omega)}^{2m}\leq \frac{\varepsilon}{2}\gamma^2 C[u_{n}(\cdot, t)]^{2ma}_{W^{s,p}(\Omega)}\left\|u_{n}(\cdot, t)\right\|_{L^2(\Omega)}^{2m(1-a)}
\\&\leq \mathcal{C}(\gamma,\varepsilon,\tilde{\varepsilon}, C)[u_{n}(\cdot, t)]^p_{W^{s,p}(\Omega)}+ \mathcal{C}(\tilde{\varepsilon})\left\|u_{n}(\cdot, t)\right\|_{L^2(\Omega)}^{\frac{2mp(1-a)}{p-2ma}} \end{split}
\end{equation}
and
\begin{equation}\begin{split}\label{FC2}
&\frac{\varepsilon_1}{2}\mu^2\left\| u_{n} (\cdot, t) \right\|^{2(q-1)} _{L^{2(q-1)}(\Omega) }\\&\leq \frac{\varepsilon_1}{2} \mu^2C_1[u_{n}(\cdot, t)]^{2(q-1)a}_{W^{s,p}(\Omega)}\left\|u_{n}(\cdot, t)\right\|_{L^2(\Omega)}^{2(q-1)(1-a)}
\\&\leq \mathcal{C}(\mu,\varepsilon_1,\tilde{\varepsilon}_1, C_1)[u_{n}(\cdot, t)]^p_{W^{s,p}(\Omega)}+ \mathcal{C}(\tilde{\varepsilon}_1)\left\|u_{n}(\cdot, t)\right\|_{L^2(\Omega)}^{\frac{2p(q-1)(1-a)}{p-2(q-1)a}}.  \end{split}
\end{equation}
Hence, from the last inequalities  \eqref{G80} we obtain
\begin{equation*}\label{G011}\begin{split}
\frac{1}{2}\|\mathcal{D}_{0|t}^{\alpha }&u_{n}(\cdot, t)\|^2_{L^2(\Omega)}+ \frac{1}{p}\mathcal{D}_{0|t}^{\alpha }[u_{n}(\cdot, t)]^p_{W^{s,p}(\Omega)}
\\&\le \mathcal{C}(\gamma,\varepsilon,\tilde{\varepsilon}, C)[u_{n}(\cdot, t)]^p_{W^{s,p}(\Omega)}+ \mathcal{C}(\tilde{\varepsilon})\left\|u_{n}(\cdot, t)\right\|_{L^2(\Omega)}^{\frac{2mp(1-a)}{p-2ma}}\\&+C(\varepsilon) \left\|\mathcal{D}_{0|t}^{\alpha}u_{n}(\cdot, t)\right\|_{L^2(\Omega)}^{2}+\mathcal{C}(\mu,\varepsilon_1,\tilde{\varepsilon}_1, C_1)[u_{n}(\cdot, t)]^p_{W^{s,p}(\Omega)}\\&+ \mathcal{C}(\tilde{\varepsilon}_1)\left\|u_{n}(\cdot, t)\right\|_{L^2(\Omega)}^{\frac{2p(q-1)(1-a)}{p-2(q-1)a}}+C(\varepsilon_1) \left\| \mathcal{D}_{0|t}^{\alpha } u_{n}(\cdot, t) \right\| _{L^2(\Omega) }^{2}.
\end{split}\end{equation*}
After choosing the constants $\varepsilon, \varepsilon_1$ such that $\large\displaystyle1>\frac{1}{\varepsilon}+\frac{1}{\varepsilon_1}$, and from the estimate \eqref{L2} it follows that
\begin{equation}\label{ES1}\begin{split}
\|\mathcal{D}_{0|t}^{\alpha }u_{n}(\cdot, t)\|^2_{L^2(\Omega)}&+ \frac{1}{p}\mathcal{D}_{0|t}^{\alpha }[u_{n}(\cdot, t)]^p_{W^{s,p}(\Omega)}
\\&\leq \mathcal{C_*}[u_{n}(\cdot, t)]^p_{W^{s,p}(\Omega)}+B(T),
\end{split}\end{equation}
where $\mathcal{C_*}:=\mathcal{C}(\gamma,\varepsilon,\tilde{\varepsilon}, C)+\mathcal{C}(\mu,\varepsilon_1,\tilde{\varepsilon}_1, C_1)$ and $  B(T):=A(\tilde{\varepsilon},T)+A(\tilde{\varepsilon}_1,T).$
Therefore,
\begin{equation}\label{Q1}\begin{split}
&\mathcal{D}_{0|t}^{\alpha }[u_{n}(\cdot, t)]^p_{W^{s,p}(\Omega)}
\le \mathcal{C_*}(p)[u_{n}(\cdot, t)]^p_{W^{s,p}(\Omega)}+ B(p,T).
\end{split}\end{equation}
Define $y(t):=[u_{n}(\cdot, t)]^p_{W^{s,p}(\Omega)}$ and using the left Riemann-Liouville integral $I_{0|t}^\alpha$ to \eqref{Q1}, according to Property \ref{ID}, we arrive at 
\begin{equation*}\label{Q3}\begin{split}
&y(t)\le y(0)+\frac{1}{\Gamma(\alpha)}\int_0^t (t-s)^{\alpha-1}\left[\mathcal{C_*}(p)y(s)+B(p,T)\right]ds,
\end{split}\end{equation*}which satisfies (see \cite{TISDELL}, Lemma 3.1) 
\begin{equation*} \label{Q4}
y(t)\leq y(0)E_{\alpha,1} (\mathcal{C_*}(p)t^{\alpha } )+\frac{B(T)}{\mathcal{C_*}}[E_{\alpha,1}(\mathcal{C_*}(p)t^\alpha)-1]:=E(p,T).
\end{equation*}
Finally, we have
\begin{equation}\label{MT}
[u_{n}(\cdot, t)]^p_{W^{s,p}(\Omega)}\leq E(p,T)\,\,\,\text{for all}\,\,\,t\in[0,T].
\end{equation}
From the inequalities \eqref{ES1} and \eqref{MT}, we obtain
\begin{equation}\label{ES2}\begin{split}
\|\mathcal{D}_{0|t}^{\alpha }u_{n}(\cdot, t)\|^2_{L^2(\Omega)}&+ \mathcal{D}_{0|t}^{\alpha }[u_{n}(\cdot, t)]^p_{W^{s,p}(\Omega)}
\\&\leq \mathcal{C_*}E(p,T)+B(p,T):=L(p,T).
\end{split}\end{equation}
Integrating both sides of \eqref{ES2} by the left Riemann-Liouville integral $I_{0|t}^\alpha$ and using Property \ref{ID}, the last inequality becomes 
\begin{equation}\label{ES3}\begin{split}
I_{0|t}^\alpha\|\mathcal{D}_{0|t}^{\alpha }u_{n}(\cdot, t)\|^2_{L^2(\Omega)}&+ [u_{n}(\cdot, t)]^p_{W^{s,p}(\Omega)}
\\&\le [u_{n}(\cdot, 0)]^p_{W^{s,p}(\Omega)}+I_{0|t}^\alpha \left[L(p,T)\right].
\end{split}\end{equation}
Consequently, applying the left Caputo derivative $\mathcal{D}_{0|t}^\alpha$  due to Property \ref{ID}, also noting the facts that $[u_{n}(\cdot, t)]^p_{W^{s,p}(\Omega)}$ is bounded and $\mathcal{D}_{0|t}^\alpha[u_{n}(\cdot, 0)]_{W^{s,p}(\Omega)}=0$, we can establish  
\begin{equation}\label{ES03}\begin{split}
\|\mathcal{D}_{0|t}^{\alpha }u_{n}(\cdot, t)\|^2_{L^2(\Omega)}
\le L(p,T),
\end{split}\end{equation}
where $L(p,T)$ does not dependent to $n$.
\newline $\bullet$ The case $\gamma>0$ and $\mu\leq0$. Accordingly, the inequality \eqref{G80} becomes
 \begin{equation*}\begin{split}
\|\mathcal{D}_{0|t}^{\alpha }u_{n}\|^2_{L^2(\Omega)}&+ \frac{1}{p}\mathcal{D}_{0|t}^{\alpha }[u_{n}(\cdot, t)]^p_{W^{s,p}(\Omega)}\leq\gamma\int_\Omega|u_n|^{m-1}u_n\mathcal{D}_{0|t}^{\alpha }u_{n}dx.   
\end{split}\end{equation*}
From the estimates \eqref{G9} and \eqref{FC1} we can rewrite the last inequality in the form
\begin{equation*}\begin{split}
\frac{1}{2}\|\mathcal{D}_{0|t}^{\alpha }u_{n}(\cdot, t)\|^2_{L^2(\Omega)}&+ \frac{1}{p}\mathcal{D}_{0|t}^{\alpha }[u_{n}(\cdot, t)]^p_{W^{s,p}(\Omega)}
\\&\le \mathcal{C}(\gamma,\varepsilon,\tilde{\varepsilon}, C)[u_{n}(\cdot, t)]^p_{W^{s,p}(\Omega)}+ \mathcal{C}(\tilde{\varepsilon})\left\|u_{n}(\cdot, t)\right\|_{L^2(\Omega)}^{\frac{2mp(1-a)}{p-2ma}}\\&+C(\varepsilon) \left\|\mathcal{D}_{0|t}^{\alpha}u_{n}(\cdot, t)\right\|_{L^2(\Omega)}^{2}.
\end{split}\end{equation*}
By choosing $\varepsilon$ small enough such that $\large\displaystyle\frac{1}{2}-C(\varepsilon)>0$, and using \eqref{L2} it follows that
\begin{equation*}\begin{split}
\|\mathcal{D}_{0|t}^{\alpha }u_{n}(\cdot, t)\|^2_{L^2(\Omega)}&+ \frac{1}{p}\mathcal{D}_{0|t}^{\alpha }[u_{n}(\cdot, t)]^p_{W^{s,p}(\Omega)}
\\&\leq \mathcal{C}(\gamma,\varepsilon,\tilde{\varepsilon}, C)[u_{n}(\cdot, t)]^p_{W^{s,p}(\Omega)}+\mathcal{C}(\tilde{\varepsilon})A(T).
\end{split}\end{equation*}
The conclusion can be obtained, as in the previous case.

\hfill \break $\bullet$ The case $\gamma\leq0$ and $\mu>0$. The inequality \eqref{G80} becomes
 \begin{equation*}\begin{split}
\|\mathcal{D}_{0|t}^{\alpha }u_{n}\|^2_{L^2(\Omega)}&+ \frac{1}{p}\mathcal{D}_{0|t}^{\alpha }[u_{n}(\cdot, t)]^p_{W^{s,p}(\Omega)}\leq\mu\int_\Omega|u_n|^{q-2}u_n\mathcal{D}_{0|t}^{\alpha }u_{n}dx.   
\end{split}\end{equation*}
Using the estimates \eqref{G10} and \eqref{FC2} we have
\begin{equation*}\begin{split}
\frac{1}{2}&\|\mathcal{D}_{0|t}^{\alpha }u_{n}(\cdot, t)\|^2_{L^2(\Omega)}+ \frac{1}{p}\mathcal{D}_{0|t}^{\alpha }[u_{n}(\cdot, t)]^p_{W^{s,p}(\Omega)}
\\&\leq \mathcal{C}(\mu,\varepsilon_1,\tilde{\varepsilon}_1, C_1)[u_{n}(\cdot, t)]^p_{W^{s,p}(\Omega)}+ \mathcal{C}(\tilde{\varepsilon}_1)\left\|u_{n}(\cdot, t)\right\|_{L^2(\Omega)}^{\frac{2p(q-1)(1-a)}{p-2(q-1)a}}\\&+C(\varepsilon_1) \left\| \mathcal{D}_{0|t}^{\alpha } u_{n}(\cdot, t) \right\| _{L^2(\Omega) }^{2}.\end{split}\end{equation*}
Taking the constant $\varepsilon_1$ small enough such that $\large\displaystyle\frac{1}{2}-C(\varepsilon_1)>0$, and noting \eqref{L2} it follows that
\begin{equation*}\begin{split}
\|\mathcal{D}_{0|t}^{\alpha }u_{n}(\cdot, t)\|^2_{L^2(\Omega)}&+ \frac{1}{p}\mathcal{D}_{0|t}^{\alpha }[u_{n}(\cdot, t)]^p_{W^{s,p}(\Omega)}
\\&\le \mathcal{C}(\mu,\varepsilon_1,\tilde{\varepsilon}_1, C_1)[u_{n}(\cdot, t)]^p_{W^{s,p}(\Omega)}+\mathcal{C}(\tilde{\varepsilon}_1)A(T).
\end{split}\end{equation*}
The conclusion of this case also can be obtained, as in the first case.
\newline $\bullet$ The case $\gamma,\mu\leq0$. Then, the estimate \eqref{G80} can rewritten as
 \begin{equation*}\begin{split}
\|\mathcal{D}_{0|t}^{\alpha }u_{n}\|^2_{L^2(\Omega)}&+ \frac{1}{p}\mathcal{D}_{0|t}^{\alpha }[u_{n}(\cdot, t)]^p_{W^{s,p}(\Omega)}\leq0.   
\end{split}\end{equation*}
Applying the left Riemann-Liouville integral $I^\alpha_{0|t}$ to the last inequality from Property \eqref{ID} it follows that 
 \begin{equation*}\begin{split}
I^\alpha_{0|t}\|\mathcal{D}_{0|t}^{\alpha }u_{n}\|^2_{L^2(\Omega)}&+ \frac{1}{p}[u_{n}(\cdot, t)]^p_{W^{s,p}(\Omega)}\leq\frac{1}{p}[u_{n}(\cdot, 0)]^p_{W^{s,p}(\Omega)}.   
\end{split}\end{equation*}
From the estimate \eqref{MT} we arrive at
\begin{equation*}\begin{split}
I^\alpha_{0|t}\|\mathcal{D}_{0|t}^{\alpha }u_{n}\|^2_{L^2(\Omega)}&\leq\frac{1}{p}[u_{n}(\cdot, 0)]^p_{W^{s,p}(\Omega)}.   
\end{split}\end{equation*}
Next, using the left Riemann-Liouville fractional derivative for the last inequality, and from Property \ref{DI} and the identity 
$$\mathbb{D}_{0|t}^{\alpha } [C]=\frac{C}{\Gamma(1-\alpha)}t^{-\alpha},$$
it follows that
\begin{equation*}\begin{split}
\|\mathcal{D}_{0|t}^{\alpha }u_{n}\|^2_{L^2(\Omega)}&\leq\frac{t^{-\alpha}}{p\Gamma(1-\alpha)}[u_{n}(\cdot, 0)]^p_{W^{s,p}(\Omega)},\,\,\,\text{for all}\,\,\,t\in[0,T].  
\end{split}\end{equation*}
Passing to the limit where $n\to\infty$, from the estimates in the previous estimates, we conclude that
\begin{equation}\label{3.1}\left\{\begin{array}{l}
u_n\in W^{s,p}(\Omega)\cap L^2(\Omega;L^\infty(0,T)),\\{}\\
\mathcal{D}_{0|t}^{\alpha }u_n\in L^2(\Omega;L^\infty(0,T)).\end{array}\right.\end{equation}
Consequently, from \eqref{3.1} there exists a subsequence $\{u_{n_k}\}$ of $\{u_n\}_{n\in\mathbb{N}}$ weak star converging to some element from $W^{s,p}(\Omega)\cap L^2(\Omega;L^\infty(0,T))$ such as  
\begin{equation}\label{G12}\begin{split}
 u_{n_k}\overset{*}{\rightharpoonup} u\,\,\,\text{in}\,\, W^{s,p}(\Omega)\cap L^2(\Omega;L^\infty(0,T)),\\\mathcal{D}_{0|t}^{\alpha }u_{n}\overset{*}{\rightharpoonup} \mathcal{D}_{0|t}^{\alpha }u\,\,\, \text{in}\,\, L^2(\Omega;L^\infty(0,T)).    
\end{split}\end{equation}
Similarly, from \eqref{G12}, we deduce that one can extract a subsequence $\{u_{n_k}\}$ of $\{u_n\}_{n\in\mathbb{N}}$ such that
\begin{equation}\label{G14}
u_{n_k}\overset{*}{\rightharpoonup}u_n \,\,\,\text{in}\,\,\, W^{s,p}(\Omega)\cap L^2(\Omega; L^\infty(0,T)). 
\end{equation}
Since, $W^{s,p}(\Omega)\cap L^2(\Omega ;L^\infty(0,T))\subset L^2(\Omega;L^\infty(0,T))$, from \eqref{3.1} it follows that the sequences $\{u_n\}_{n\in\mathbb{N}}$ and $\mathcal{D}_{0|t}^{\alpha }u_{n}$ are bounded in $L^2(\Omega;L^\infty(0,T))$. Then, it particular $\{u_n\}_{n\in\mathbb{N}}$ is bounded in $W^{s,p}(\Omega)$. It is known by Lemma \ref{EMBD}, that the embedding of $W^{s,p}(\Omega)$ in $L^2(\Omega)$ is continuous. It gives us that  the subsequence $\{u_{n_k}\}$  can be chosen such that $u_{n_k}\to u$ in the norm of $L^2(\Omega)$, converging almost everywhere.
The previous argument leads us to the limit in \eqref{G2}. However, we multiply \eqref{G2} by $\theta_k(t)\in C[0,T]$, then summing up both sides over $k=\overline{1,n}$, to get
\begin{equation*}\label{G014}\begin{split}
&\int_\Omega\mathcal{D}_{0|t}^{\alpha }u_{n}\cdot \Psi dx+ P(u_{n},\Psi)
=\gamma\int_\Omega|u_n|^{m-1}u_n\cdot \Psi dx+\mu\int_\Omega|u_n|^{q-2}u_n\cdot \Psi dx,   
\end{split}\end{equation*}
almost everywhere in $t\in [0,T]$, where $\large\displaystyle\Psi (x,t)=\sum _{k=1}^{n}\theta_{k} (t)\omega _{k} (x) $.

Taking into account the obtained inclusions and convergence, we pass in \eqref{G2} to the limit as $n\to \infty $ and obtain Definition \ref{weak} for $\varphi =\Psi $. Since the set of all functions $\Psi (x,t)$ is dense in $\Pi$, then the limit relation holds for all $\varphi=\varphi (x,t)\in W_0^{s,p}(\Omega;L^p(0,T)).$
\\$\bullet$ The case $1<q-1<m <p.$ 
We repeat the entire procedure described above by simply changing the condition inequality \eqref{INQ1} to $1<q-1<m<p.$
\proofend

\subsection{ Uniqueness of a weak solution} 
In this subsection we discuss the uniqueness of weak solutions.
\begin{theorem}
Let  $u_0\in W^{s,p}_0(\Omega), u_0\geq0$ and $sp<N$. Then the local real-valued weak solution of \eqref{01} on $(0,T),$ $T<\infty$, is unique.
\end{theorem}
\proof
Assume that we have two real-valued weak solutions $u$ and $v$ for problem \eqref{01}. Hence, by Definition \ref{weak}, we obtain 
\begin{equation*}
\begin{split}
\int_0^T\int_{\Omega}\mathcal{D}_{0|t}^\alpha u\varphi dxdt&\!+\!\int_0^T\int_{\Omega}\frac{|u(x)\!-\!u(y)|^{p-2}(u(x)\!-\!u(y))}{|x-y|^{N+sp}}(\varphi(x)\!-\!\varphi(y))dxdydt\\&=\gamma\int_0^T\int_{\Omega}|u|^{m-1}u\varphi dxdt+\mu\int_0^T\int_{\Omega}|u|^{q-2}u\varphi dxdt    
\end{split}     
\end{equation*}
and
\begin{equation*}
\begin{split}
\int_0^T\int_{\Omega}\mathcal{D}_{0|t}^\alpha v\varphi dxdt&+\int_0^T\int_{\Omega}\frac{|v(x)\!-\!v(y)|^{p-2}(v(x)\!-\!v(y))}{|x-y|^{N+sp}}(\varphi(x)\!-\!\varphi(y))dxdydt\\&=\gamma\int_0^T\int_{\Omega}|v|^{m-1}v\varphi dxdt+\mu\int_0^T\int_{\Omega}|v|^{q-2}v\varphi dxdt.    
\end{split}     
\end{equation*}
By subtracting the previous two inequalities, it follows for $t\in(0,T]$ that
\begin{equation*}\begin{split}\label{UU1}
&\int_0^t\int_{\Omega}\mathcal{D}_{0|\tau}^\alpha [u-v]\varphi dxd\tau+\underbrace{\int_0^t\int_{\Omega}[(-\Delta)^s_pu-(-\Delta)^s_pv]\varphi dxd\tau}_{\mathcal{C}}\\ & =\underbrace{\gamma\int_0^t\int_{\Omega}(|u|^{m-1}u-|v|^{m-1}v)\varphi dxd\tau}_{\mathcal{A}}+\underbrace{\mu\int_0^t\int_{\Omega}(|u|^{q-2}u-|v|^{q-2}v)\varphi dxd\tau}_{\mathcal{B}}.\end{split}\end{equation*} 
Using the fact that $\mathcal{C}$ is nonnegative from \eqref{I1}, and the estimates \eqref{E12}, \eqref{E13} for $\mathcal{A}, \mathcal{B}$, respectively, we deduce that
\begin{equation*}\begin{split}
\int_0^t\int_{\Omega}\mathcal{D}_{0|\tau}^\alpha [u-v]\varphi dxd\tau&\leq \gamma L(m)\int_0^t\int_{\Omega}|u-v|\varphi dxd\tau\\&+\mu L(q)\int_0^t\int_{\Omega}|u-v|\varphi dxd\tau.\end{split}\end{equation*} 
At this stage choosing the real-valued test function $$\varphi=(u-v)_+=\max\{u-v, 0\}$$ and using Lemma \ref{AA}, we can rewrite the last inequality as
\begin{equation*}\begin{split}\label{UN}
&\frac{1}{2}\int_0^t\int_{\Omega}\mathcal{D}_{0|\tau}^\alpha(u-v)_+^2 dxd\tau\\&\leq \gamma  L(m)\int_0^t\int_{\Omega}(u-v)_+^2 dxd\tau+\mu L(q)\int_0^t\int_{\Omega}(u-v)_+^2 dxd\tau.\end{split}\end{equation*} 
Therefore, we should consider three cases depending on $\gamma, \mu$. By repeating the entire procedure as in the proof of Theorem \ref{CP}, we obtain the main inequality
\begin{equation*}\begin{split}
\int_{\Omega}(u-v)_+^2 dx\leq0,\end{split}\end{equation*}which is equivalent to $(u-v)_+=0$. Finally, we conclude that $u=v$.\proofend

\section{Global existence and blow-up of solutions}
\setcounter{section}{5} \setcounter{equation}{0}
\subsection{Blow-up of solution} In this subsection we will show the blow-up of solution to \eqref{01} using the comparison principle.

Let $\xi(x)>0$ and $\lambda_1(\Omega)>0$ be the first eigenfunction and the first eigenvalue \cite[Theorem 5]{Lindgren}, respectively, related to the Dirichlet problem:
\begin{equation}\label{EF01} 
\left\{\begin{array}{l}
(-\Delta)^s_p\xi(x)=\lambda_1(\Omega) |\xi(x)|^{p-2}\xi(x),\,\,\, x\in\Omega,\\{}\\
\xi(x)=0,\,\,\, x\in \mathbb{R}^N\setminus\Omega, \end{array}\right.\end{equation}
with $\|\xi\|^2_{L^2(\Omega)}= 1.$
\begin{theorem}\label{BB}
Let $p\geq 2,$ $u_0>0$, and assume that one of the following conditions holds:\\
$(a)$ $p=q\geq2, m>1$ and $\lambda_1(\Omega)\geq\mu, \gamma>0$;\\
$(b)$ $p-1=m\geq1, q>2$ and $ \lambda_1(\Omega)\geq\gamma, \mu>0$;\\
$(c)$ $p\geq2, m>1, q\geq1$ and $\lambda_1(\Omega),\gamma>0, \mu\leq 0$;\\
$(d)$ $p\geq2, m+1=q>2$ and $\gamma, \mu, \lambda_1(\Omega)>0.$\\
Then the positive solution $u(x,t)$ of \eqref{01} blows up in finite time 
$$T^*=(k\Gamma(2-\alpha))^{\frac{2}{2-2\alpha-k}},$$
where $k=m-1$ in cases (a), (c), (d) and $k=q-2$ in cases (b), (d) and $\alpha\in(0,1)$, $\Gamma$ is the Euler Gamma function, namely, we have $$\lim\limits_{t\rightarrow T^*}u(x,t)=+\infty.$$
\end{theorem}
\proof
First we will prove the cases (a) and (b). 

We shall prove this theorem by constructing a proper weak subsolution to \eqref{01}.
We will seek the solution $v(x,t)=\xi(x)f(t)>0$ with the initial data $v_0(x)=\xi(x)f(0),$ such that $0\leq v_0(x)\leq u_0(x)$ on $x\in\Omega$.  
Multiplying the equation \eqref{01} by $v(x,t)$, and integrating
the equality over $\Omega$, one obtains
 \begin{equation*}\label{xf}\begin{split}
f(t)\mathcal{D}_{0|t}^{\alpha }f(t)\|\xi\|^2_{L^2(\Omega)}&+\lambda_1(\Omega) f^{p}(t)\|\xi^{p}\|^2_{L^2(\Omega)}\\&=\gamma f^{m+1}(t)\|\xi^{m+1}\|^2_{L^2(\Omega)}+\mu f^{q}(t)\|\xi^{q}\|^2_{L^2(\Omega)}.     
 \end{split}\end{equation*}
Hence, from Lemma \ref{AA}, it follows that
 \begin{equation}\label{xf1}\begin{split}
\frac{1}{2}\mathcal{D}_{0|t}^{\alpha }f^2(t)&+\lambda_1 C(p) f^{p}(t)\leq\gamma C(m)f^{m+1}(t)+\mu C(q)f^{q}(t).  \end{split}\end{equation}
At this stage, by denoting $f^2(t)=z(t)$, we have to consider the cases:

$(a)$ If $p=q \geq 2, m>1$ and $\lambda_1(\Omega)   \geq \mu, \gamma>0$, then \eqref{xf1} can rewritten as
 \begin{equation*}\label{xf2}\mathcal{D}_{0|t}^{\alpha }z(t)\leq2 C(m,\gamma) z^\frac{m+1}{2}(t).
\end{equation*}
Using the idea of paper \cite{Coclite}, we set for any $t\in(0,b)$,
\begin{equation*}\label{xf3}z(t)=\frac{b}{(b-t)^\frac{2}{m-1}},\,\,\,b:=b(m-1,\alpha)=\left((m-1)\Gamma(2-\alpha)\right)^\frac{2}{3-2\alpha-m}.
\end{equation*}
Accordingly, we have the initial condition $z(0)=z_0>0$.
We should note that the function $z(t)$, 
$\lim_{t\to b^-}z(t)\to\infty,$ diverges at $t=b$.
Moreover, for any $t\in(0,b)$ and any $\tau\in(0,t)$ we can obtain
\begin{align*}\label{xf4}\frac{\partial}{\partial\tau}z(\tau):&=\frac{2b}{(m-1)(b-\tau)^\frac{m+1}{m-1}}\\&\leq\frac{2b}{(m-1)(b-t)^\frac{m+1}{m-1}}\\&=\frac{z^\frac{m+1}{2}(t)}{2(m-1)b^\frac{m-1}{2}}.
\end{align*}
From Definition \ref{CD} it follows for all $t\in(0,b)$,
\begin{equation*}\label{xf5}\begin{split}
\mathcal{D}^\alpha_{0|t}z(t)&=\frac{2 C(m,\gamma)}{\Gamma(1-\alpha)}\int_0^t\frac{z'(\tau)}{(t-\tau)^\alpha}d\tau
\\&\leq\frac{2 C(m,\gamma) z^\frac{m+1}{2}(t)}{2(m-1)b^\frac{m-1}{2}\Gamma(1-\alpha)}\int_0^t\frac{d\tau}{(t-\tau)^\alpha}\\&=\frac{C(m,\gamma)t^{1-\alpha}z^\frac{m+1}{2}(t)}{(m-1)b^\frac{m-1}{2}\Gamma(2-\alpha)}
\\&\leq\frac{ C(m,\gamma)b^{1-\alpha}z^\frac{m+1}{2}(t)}{(m-1)b^\frac{m-1}{2}\Gamma(2-\alpha)}    
\\&=\frac{ C(m,\gamma)b^\frac{3-2\alpha-m}{2}z^\frac{m+1}{2}(t)}{(m-1)\Gamma(2-\alpha)}
\\&= C(m,\gamma)z^\frac{m+1}{2}(t).
\end{split}
\end{equation*}
Therefore, $z(t)$ diverges at $t=b$ yielding that
$$T_*\leq b=b(m-1,\alpha).$$

$(b)$ If $p-1=m\geq1, q>2$ and $ \lambda_1(\Omega)\geq\gamma, \mu>0$, then from \eqref{xf1} we obtain
 \begin{equation*}\label{xf02}\mathcal{D}_{0|t}^{\alpha }z(t)\leq 2C(q,\mu)z^\frac{q}{2}(t).
\end{equation*}
We can argue as the previous case by choosing 
for any $t\in(0,b)$ the function
\begin{equation*}\label{xf6}z(t):=\frac{b}{(b-t)^\frac{2}{q-2}},\,\,\,b:=b(q-2,\alpha)=\left((q-2)\Gamma(2-\alpha)\right)^\frac{2}{4-2\alpha-q}.\end{equation*}
Similarly, for any $t\in(0,b)$ and any $\tau\in(0,t)$, we obtain
\begin{align*}\label{xf7}\frac{\partial}{\partial\tau}z(\tau):&=\frac{2b}{(q-2)(b-\tau)^\frac{q}{q-2}}\\&\leq\frac{2b}{(q-2)(b-t)^\frac{q}{q-2}}\\&=\frac{z^\frac{q}{2}(t)}{2(q-2)b^\frac{q-2}{2}}.
\end{align*}
Finally, $z(t)$ diverges at $t=b$ for
$$T_*\leq b=b(q-2,\alpha).$$

$(c)$ For $p\geq2, m>1, q\geq1$ and $\lambda_1(\Omega),\gamma>0, \mu<0$, inequality \eqref{xf1} yields
 \begin{equation*}\label{xf8}\mathcal{D}_{0|t}^{\alpha }z(t)\leq 2 C(m,\gamma) z^\frac{m+1}{2}(t).
\end{equation*}

$(d)$ For $p\geq2, m+1=q>2$ and $ \gamma, \mu, \lambda_1(\Omega)>0$, using \eqref{xf1}  we have
 \begin{equation*}\label{xf11}\mathcal{D}_{0|t}^{\alpha }z(t)\leq 2 \left[C(m,\gamma)+C(p,\mu)\right] z^\frac{q}{2}(t).
\end{equation*}
Proof of $(c)$ and $(d)$ can be derived from the previous cases. We just omit it. The proof is complete.
\proofend

\subsection{Global solution} In this subsection, we prove the existence of global solutions of problem \eqref{01}. 
\begin{theorem}\label{GS} Assume that $u_0\in W^{s,p}_0(\Omega)\cap L^\infty(\Omega),\,s\in(0,1),\, u_0\geq0$, and let $p, q, m, \gamma, \mu$ satisfy one of the following conditions:
\\
$(a)$ $p=m+1=q>2$ and $0<\gamma+\mu\leq\lambda_1(\Omega);$\\
$(b)$ $p=q$ or $p=m+1$ and  $0\leq\gamma, \mu\leq\lambda_1(\Omega);$  
\\
$(c)$ $p\leq m+q$ and $\gamma,\mu\in\mathbb{R};$
\\$(d)$ $p\geq2, m>1, q\geq1$ and $\gamma,\mu\leq0;$
\\$(e)$ $p=q,  m>1$ and $\gamma\leq0,\,\mu>0$.\\
Then the problem \eqref{01} admits a global in time positive solution.\end{theorem}
\begin{remark}
Note that in the limiting case $\alpha\rightarrow 1$ and $s\rightarrow 1,$ the results of Theorem \ref{GS} coincides with the results obtained in \cite{Li}.
\end{remark}
\proof {\it of Theorem \ref{GS}} $(a)$  Let $\Omega^*\subset\mathbb{R}^N$ be a smooth domain such that $\Omega\subset\subset\Omega^*$.

Define $\psi$ and $\lambda_1(\Omega^*)$ to be the first eigenfunction and the first eigenvalue related to the Dirichlet problem:
\begin{equation*}\label{EF1} 
\left\{\begin{array}{l}
(-\Delta)^s_p\psi(x)=\lambda_1(\Omega^*)|\psi(x)|^{p-2}\psi(x),\,\,\, x\in\Omega^*,\\{}\\
\psi(x)=0,\,\,\, x\in \mathbb{R}^N\setminus\Omega^*, \end{array}\right.\end{equation*}
with $\large\displaystyle\int_{\Omega^*}|\psi(x)|^{p}dx=1$, for more details see \cite[Lemma 15]{Lindgren}. 
Then, from Lemma \ref{LL} we have $\lambda_1(\Omega^*)\leq\lambda_1(\Omega)$, where $\lambda_1(\Omega)$ is the  first eigenvalue of \eqref{EF}. Moreover, in view of \cite[Theorem 16]{Lindgren}, we can choose a suitable $\Omega^*$ and $\theta>0$ which satisfies $\theta\leq\lambda_1(\Omega^*)\leq\lambda_1(\Omega)$ . Therefore, let $K$ be so large such that  $$w=K\psi\geq K\beta\geq\|u_0\|_{L^\infty(\Omega)},$$ where $\beta=\inf_{\Omega}\psi>0,$ which we note that $\psi>0$ in $\Omega$ from the results of Lindgren and Lindqvist in \cite[Theorem 5]{Lindgren}.
Following that, a simple calculation shows that for each nonnegative test-function $\varphi=\varphi(x,t)\in \Pi\cap W_0^{s,p}(\Omega;L^\infty(0,T))$, we have \begin{equation}\label{GES01}\begin{split}
\int_0^T\int_{\Omega}\mathcal{D}_{0|t}^{\alpha }w\varphi dxdt&+\int_0^T\langle(-\Delta)^s_pw,\varphi\rangle dt \\&=\gamma\int_0^T\int_{\Omega}w^m\varphi dxdt+\mu\int_0^T\int_{\Omega}w^{q-1}\varphi dxdt,\end{split}\end{equation}
where $\langle\cdot,\cdot\rangle$ is  the inner product.
Hence, noting that $p=m+1=q>2,$ and choosing $\theta:=\gamma+\mu$, the last identity takes the form
\begin{equation*}\label{GS4}\begin{split}\int_0^T\int_{\Omega}\mathcal{D}_{0|t}^{\alpha }w\varphi dxdt+\int_0^T\langle(-\Delta)^s_pw,\varphi\rangle dt&=\lambda_1(\Omega)\int_0^T\int_{\Omega}w^{p-1}\varphi dxdt\\&\geq\lambda_1(\Omega^*)\int_0^T\int_{\Omega}w^{p-1}\varphi dxdt \\&\geq(\gamma+\mu)\int_0^T\int_{\Omega}w^{p-1}\varphi dxdt.\end{split}\end{equation*}
It follows that $w=K\psi$ is a weak supersolution of problem \eqref{01}.  From Theorem \ref{CP}, we have $0\leq u\leq w$ almost everywhere in $\Omega_T$.
It is also important to note that the function $w$ is independent of $t$, allowing us to continue the method at any time interval $[T, T']$. As a result, we may say that the solution to \eqref{01} is global in time.

$(b)$ Due to the expression  \eqref{GES01} and the conditions $p=q$ or $p=m+1$, it follows that 
\begin{equation*}\label{GE02}\begin{split}
&\int_0^T\int_{\Omega}\mathcal{D}_{0|t}^{\alpha }w\varphi dxdt+\int_0^T\langle(-\Delta)^s_pw,\varphi\rangle dt \geq\mu\int_0^T\int_{\Omega}w^{p-1}\varphi dxdt\end{split}\end{equation*}
and
\begin{equation*}\label{GES03}\begin{split}
&\int_0^T\int_{\Omega}\mathcal{D}_{0|t}^{\alpha }w\varphi dxdt+\int_0^T\langle(-\Delta)^s_pw,\varphi\rangle dt \geq\gamma\int_0^T\int_{\Omega}w^{p-1}\varphi dxdt.\end{split}\end{equation*}
Since, $0\leq\gamma,\mu\leq\lambda_1(\Omega)$, then the function $w=K\psi$ is also a weak supersolution of \eqref{01}. The conclusion is established using the same argument as before.

$(c)$ From Definition \ref{FPLD} assume that $u$ is an eigenfunction associated to the eigenvalue $\lambda_1(\Omega)$, which is a nonnegative \cite[Theorem 5]{Lindgren}. Then by \eqref{GES01} it follows that 
\begin{equation}\label{GS02}\begin{split}
\lambda_1(\Omega)\int_0^T\int_{\Omega}u^{p-1}\varphi dxdt =\gamma\int_0^T\int_{\Omega}u^m\varphi dxdt+\mu\int_0^T\int_{\Omega}u^{q-1}\varphi dxdt.\end{split}\end{equation}
Choosing constants $r, r'$ such as
\begin{equation*}\begin{split}
\frac{1}{r}+\frac{1}{r'}=1,\,\,\,r, r'>1\,\,\,\text{and}\,\,\,p-1=\frac{m}{r}+\frac{q-1}{r'}\leq m+q-1, \end{split}\end{equation*}
we obtain
\begin{equation*}\begin{split}
\int_0^T\int_{\Omega}u^{p-1}\varphi dxdt =\int_0^T\int_{\Omega}u^{\frac{m}{r}+\frac{q-1}{r'}}\varphi^{\frac{1}{r}+\frac{1}{r'}}dxdt.\end{split}\end{equation*}
Using the H\"{o}lder and $\varepsilon$-Young inequalities to the last expression, it follows that 
\begin{equation*}\label{GES02}\begin{split}
\int_0^T\int_{\Omega}u^{p-1}\varphi dxdt &\leq\left(\int_0^T\int_{\Omega}u^m\varphi dxdt\right)^{\frac{1}{r}}\left(\int_0^T\int_{\Omega}u^{q-1}\varphi dxdt\right)^{\frac{1}{r'}}
\\&\leq\varepsilon\int_0^T\int_{\Omega}u^m\varphi dxdt+C(\varepsilon)\int_0^T\int_{\Omega}u^{q-1}\varphi dxdt.\end{split}\end{equation*}
Therefore, the identity \eqref{GS02} becomes
\begin{equation*}\label{GS03}\begin{split}
\gamma\int_0^T\int_{\Omega}u^m\varphi dxdt&+\mu\int_0^T\int_{\Omega}u^{q-1}\varphi dxdt\\&\leq\lambda_1(\Omega)\varepsilon\int_0^T\int_{\Omega}u^m\varphi dxdt+\lambda_1(\Omega) C(\varepsilon)\int_0^T\int_{\Omega}u^{q-1}\varphi dxdt.\end{split}\end{equation*}

Now, taking $\varepsilon$ small enough, such that $\lambda_1(\Omega)\varepsilon-\gamma\geq0$ and $\lambda_1(\Omega) C(\varepsilon)-\mu\geq0$, we can get that the last inequality will be non-positive
\begin{equation*}\label{GS031}\begin{split}
(\lambda_1(\Omega)\varepsilon-\gamma)\int_0^T\int_{\Omega}u^m\varphi dxdt+(\lambda_1(\Omega) C(\varepsilon)-\mu)\int_0^T\int_{\Omega}u^{q-1}\varphi dxdt\geq 0,\end{split}\end{equation*}
which completes our proof by the comparison principle.

$(d)$  We proceed by multiplying each term of \eqref{01} by $u\geq0$ and then integrating over $\Omega$. Thus,  we obtain
\begin{equation}\label{GES1}\begin{split} \int_\Omega [\mathcal{D}_{0|t}^{\alpha }u]udx&=-\langle(-\Delta)^s_pu,u\rangle+\gamma\int_\Omega u^{m+1}dx+\mu\int_\Omega u^{q}dx\end{split}
\end{equation}
and taking into account that $\gamma,\mu\leq0,$ $\langle(-\Delta)^s_pu,u\rangle\geq0$, it follows that 
\begin{equation*}\begin{split}
\int_\Omega[\mathcal{D}_{0|t}^{\alpha }u]udx&\leq0.    
\end{split}\end{equation*}
By Lemma \ref{AA}, it implies 
\begin{equation*}\begin{split}
 \int_\Omega \mathcal{D}_{0|t}^{\alpha}u^2dx\leq 0.    
\end{split}\end{equation*}
Moreover, the Caputo derivative depends on the variable $t$, and the last expression can be rewritten as
\begin{equation}\label{LLQ}\begin{split}
 \mathcal{D}_{0|t}^{\alpha}\int_\Omega u^2dx\leq 0.    
\end{split}\end{equation}
Hence, applying the left Riemann-Liouville integral $I_{0|t}^{\alpha}$ to the inequality \eqref{LLQ} and using Property \ref{ID}, we obtain
\begin{equation*}\begin{split}\int_\Omega u^2(x,t)dx\leq\int_\Omega u^2(x,0)dx.\end{split}\end{equation*}
Finally, using $u_0\geq0$ and Corollary \ref{cor1}, we get
$$\|u(\cdot,t)\|_{L^2(\Omega)}\leq\|u_0\|_{L^2(\Omega)},\,\,\,\text{for all}\,\,\,t\geq0.$$

$(e)$ Without loss of generality, for $\gamma\leq0,\,\mu>0$  we can get from \eqref{GES1}, by \eqref{PSP}, that
\begin{equation*}\begin{split} \int_\Omega [\mathcal{D}_{0|t}^{\alpha }u]udx&\leq-C_{N,s,p}\int_{\Omega}\int_{\Omega}\frac{|u(x)-u(y)|^{p}}{|x-y|^{N+sp}}dxdy+\mu\int_\Omega u^{q}dx.\end{split}
\end{equation*}
Then, by Lemma \ref{Sharp}  for $p=q$, we obtain
\begin{equation*}\begin{split} \int_\Omega [\mathcal{D}_{0|t}^{\alpha }u]udx&\leq-C_{N,s,p}[u]^p_{W^{s,p}(\Omega)}+ \mu\lambda_1(\Omega)[u]^p_{W^{s,p}(\Omega)}.\end{split}
\end{equation*}
Using the fact that  $\lambda_1(\Omega)$ coincides with the sharp constant in Lemma \ref{Sharp} \cite[page 2]{Brasco3} we choose the domain such that  $\large\displaystyle C_{N,s,p}\geq\mu\lambda_1(\Omega)\geq\frac{\mu}{\mathcal{I}_{N,s,p(\Omega)}}$ holds, which gives us
\begin{equation*}\begin{split} \int_\Omega [\mathcal{D}_{0|t}^{\alpha }u]udx&\leq0.\end{split}
\end{equation*}
Accordingly, the conclusion follows as in the previous case. \proofend

\subsection{Asymptotic behavior of solution}
In this subsection, we give the time-decay estimates of global solutions
of problem \eqref{01}.
\begin{theorem}\label{DS} Assume that $u_0>0$ and that one of the following conditions holds:\\
$(a)$ $m=q-1>0$ and $\gamma+\mu< 0;$\\
$(b)$ $m>0, q>1$ and  $\gamma< 0,\,\mu=0;$\\
$(c)$ $m>0, q>1$ and  $\gamma=0,\,\mu< 0$.\\
Then the positive global solution to problem \eqref{01} satisfies the estimate
\begin{equation*}
0<u(x,t)\leq \frac{M}{1+t^\frac{\alpha}{r}},\,t\geq 0,\,x\in\Omega,
\end{equation*} where $M$ is a positive constant dependent of $u_0,$ and $r=m$ in cases (a), (b) and $r=q-1$ in cases (a), (c).
\end{theorem}
\proof
(a) Let us consider the function $v(x,t):=v(t)>0$ for all $x\in\overline{\Omega}$. Then it follows that
$$\mathcal{D}_{0|t}^{\alpha }v(t)+(-\Delta)^s_pv(t)=\gamma v^{m}(t)+\mu v^{q-1}(t).$$

According to the fact that $(-\Delta)^s_pv(t)=0$ and $m=q-1>0$, $\gamma+\mu< 0$, the last expression can be rewritten in the following form  
\begin{equation}\label{Eq*}\mathcal{D}_{0|t}^{\alpha }v(t)+\nu v(t)^{m}=0,\,\nu=-(\gamma+\mu)>0,
\end{equation}
which ensures that $v(t)$ satisfies \eqref{01} with the initial data $0<\max\limits_{x\in\Omega}u_0(x)\leq v_0$.

It is known from the results of Zacher and Vergara in \cite[Theorem 7.1]{Vergara2}, that if $v_0>0, \nu>0, m>0$, then the solution to equation \eqref{Eq*} satisfies estimate $v(t)\leq\frac{M}{1+t^\frac{\alpha}{r}},$ for all $t\geq 0.$ As $0<u_0(x)\leq v_0,$ then $v(t)$ is a supersolution of problem \eqref{01}. This completes the proof.

Cases (b) and (c) are proved in a similar way, completely repeating the above calculations. 

The proof is complete.
\proofend

\section*{\small
 Conflict of interest} 

 {\small
 The authors declare that they have no conflict of interest.}

\begin{acknowledgements}
This research has been funded by the Science Committee of the Ministry of Education and Science of the Republic of Kazakhstan (Grant No. AP14972726) and by the FWO Odysseus 1 grant G.0H94.18N: Analysis and Partial Differential Equations. Michael Ruzhansky was supported by the EPSRC grant EP/R003025/2 and by the Methusalem programme of the Ghent University Special Research Fund (BOF)(Grant number 01M01021).
\end{acknowledgements}



\end{document}